\documentclass[a4paper,12pt,oneside]{article}

\usepackage[utf8]{inputenc}
\usepackage[english]{babel}
\usepackage[T1]{fontenc}           	
\usepackage{amsmath}
\usepackage{amssymb}
\usepackage{amsfonts}          		
\usepackage[all]{xy}
\usepackage{amscd}
\usepackage{amsthm}			
\usepackage{algorithm2e} 
\usepackage{color}
\usepackage{psfrag}
\usepackage{url}
\usepackage{hyperref}
\usepackage[nottoc]{tocbibind}
\usepackage{framed,color}
\usepackage[new]{old-arrows}
\usepackage[]{listings}
\usepackage[]{float}
\usepackage{hyperref}
\usepackage[dvips,final]{graphicx}
\usepackage[dvips]{geometry}
\usepackage{epsfig}
\usepackage{latexsym}
\usepackage{bm}
\usepackage{comment}
\usepackage{xfrac}
\usepackage{microtype}
\usepackage{epstopdf}
\usepackage[table,x11names]{xcolor}
\usepackage{lmodern}
\usepackage{textcomp}
\usepackage{tipa}
\usepackage{enumitem}

\usepackage[all]{pstricks}
\usepackage{caption}
\usepackage{subcaption}
\usepackage{mathtools}

\newcounter{count}[section]

\theoremstyle{plain}
\newtheorem{theorem}[count]{Theorem}
\newtheorem{prop}[count]{Proposition}

\newtheorem{lemma}[count]{Lemma}
\newtheorem{example}[count]{Example}

\newtheoremstyle{named}{}{}{\itshape}{}{\bfseries}{.}{ }{#1 \thmnote{#3}}
\theoremstyle{named}
\newtheorem*{namedtheorem}{Theorem}

\theoremstyle{definition}
\newtheorem{definition}[count]{Definition}

\newtheorem{remark}[count]{Remark}

\newtheorem{conj}[count]{Conjecture}
\newtheorem{question}[]{Question}

\newcommand{\Z}{\mathbb{Z}}
\newcommand{\Q}{\mathbb{Q}}

\newcommand{\D}{\mathbb{D}}

\DeclareMathOperator{\rank}{rank}

\DeclarePairedDelimiter{\tonde}{(}{)}

\newcommand{\nota} [1] {\caption{\footnotesize{#1}}}

\newcommand{\qhs}{$\mathbb{Q}HS$}
\newcommand{\qht}{$\mathbb{Q}HT$}

\newcommand\blfootnote[1]{%
  \begingroup
  \renewcommand\thefootnote{}\footnote{#1}%
  \addtocounter{footnote}{-1}%
  \endgroup
}

\lstdefinestyle{mystyle}{ 
    breakatwhitespace=false,         
    breaklines=true,                 
    captionpos=b,                    
    keepspaces=true,                 
    numbers=left,                    
    numbersep=5pt,                  
    showspaces=false,                
    showstringspaces=false,
    showtabs=false,                  
    tabsize=2
}

\usepackage{ mathrsfs }

\begin{document}
\title{Dodecahedral $L$-spaces and hyperbolic 4-manifolds}
\author{Ludovico Battista, Leonardo Ferrari, Diego Santoro}
\date{}

\maketitle

\begin{abstract}
We prove that exactly 6 out of the 29 rational homology 3-spheres tessellated by four or less right-angled hyperbolic dodecahedra are $L$-spaces. The algorithm used is based on the $L$-space census provided by Dunfield in \cite{D}, and relies on a result by Rasmussen-Rasmussen \cite{RR}. We use the existence of these manifolds together with a result of Martelli \cite{M} to construct explicit examples of hyperbolic 4-manifolds containing separating $L$-spaces, and therefore having vanishing Seiberg-Witten invariants. This answers a question asked by Agol and Lin in \cite{AL}.\blfootnote{L.F.\ was supported by the Swiss National Science Foundation, project no.\ PP00P2–202667.}
    
\end{abstract}

\section{Introduction}

In recent years, the field of low-dimensional topology has seen a growing interest in the study of $L$-spaces (see Definition \ref{defn:Lspace}).
Although determining if a given manifold is an $L$-space is algorithmically decidable (see \cite{SarWan}), it seems difficult to implement such an algorithm. A great work in this direction was done by Dunfield in \cite{D}, where the author identifies all the $L$-spaces in a census $\mathscr{Y}$ of $\sim 300'000$ hyperbolic 3-manifolds, through the use of an important tool provided by \cite{RR} (see Theorem \ref{theorem RR}).

The procedure in \cite{D} makes large use of the classification \cite{B} of the $1$-cusped hyperbolic $3$-manifolds that can be triangulated with $9$ or less ideal tetrahedra; as a consequence, the manifolds in $\mathscr{Y}$ have volume bounded by 9.14. 

In this paper, we determine the $L$-spaces among some dodecahedral manifolds. 

\begin{definition}\label{defn:dodecahedral_manifold}
A hyperbolic 3-manifold is \emph{dodecahedral} if it may be tessellated by regular right-angled hyperbolic dodecahedra.
\end{definition}

Throughout the paper, all the manifolds are smooth, connected and orientable, unless otherwise stated. Hyperbolic manifolds are always complete.

The dodecahedral manifolds tessellated with four or less dodecahedra were classified in \cite{Goerner}. Using this, we give the following definition:

\begin{definition}\label{defn:censusD}
We denote by $\mathscr{D}$ the set of the 29 dodecahedral hyperbolic rational homology spheres tessellated with four or less dodecahedra (see Table \ref{table:dodecahedralmanifolds}).
\end{definition}

To identify the $L$-spaces in $\mathscr{D}$, we elaborate on some ideas presented by Dunfield in \cite{D}, and we use the algorithm described in Section \ref{section:proving_L_space}. With the help of the code provided by Dunfield in \cite{D}, we show that the remaining 3-manifolds are not $L$-spaces, so we can conclude:

\begin{theorem}\label{thm:classdodman}
Among the 29 manifolds in $\mathscr{D}$, 6 are $L$-spaces and 23 are not; see Tables \ref{table:dodecahedralmanifolds} - \ref{table:dodecahedralmanifoldswithorderable}.
\end{theorem}

The information given by Theorem \ref{thm:classdodman} is very little compared with the one from \cite{D}. Nevertheless, the geometric properties of the manifolds in $\mathscr{D}$ can be used to answer a question asked by Agol and Lin in \cite{AL}. Before stating the question, we give a brief introduction to the problem. 

Seiberg-Witten invariants are smooth invariants for $4$-manifolds with $b_2^+\geq2$ and were defined in \cite{SW,SW1,W} by Seiberg and Witten. These invariants, coming from gauge theory, soon established surprising connections between the topology and the geometry of smooth $4$-manifolds. For example, if a $4$-manifold with $b_2^+\geq 2$ has a metric with positive scalar curvature then these invariants all vanish \cite{W}, while on the other side Taubes \cite{Tau} proved that a symplectic $4$-manifold with $b_2^+\geq2$ has a non-zero Seiberg-Witten invariant. Putting together these results, we have that any symplectic $4$-manifold with $b_2^+\geq 2$ does not admit a metric with positive scalar curvature.

In \cite{L} LeBrun conjectured that the Seiberg-Witten invariants of a closed hyperbolic $4$-manifold are all zero. In \cite{AL} Agol and Lin showed the existence of infinitely many commensurability classes of hyperbolic $4$-manifolds containing representatives with vanishing Seiberg-Witten invariants. This is shown by proving that there exist hyperbolic $4$-manifolds that contain separating $L$-spaces. Part of their proof was based on a result regarding the embeddings of arithmetic hyperbolic manifolds proved by Kolpakov-Reid-Slavich \cite{KRS}. As a consequence of this the hyperbolic $4$-manifolds of \cite{AL} are not explicitly constructed. Therefore they asked the following:

\begin{question}[{\cite[Conclusions (1)]{AL}}]\label{qst1}
Can one find an explicit hyperbolic $4$-manifold $N$ such that $N=N_1\cup_{M'} N_2$, where the separating hypersurface $M'$ is an $L$-space and such that $b_2^+(N_i)\geq 1$ for $i=1,2$?
\end{question}

The separating hypersurfaces that we will use are built from the ones in $\mathscr{D}$ and to build the 4-manifold we will follow the construction presented in \cite{M}. We discuss all the details in Section \ref{sec:building_4_man}. Here we just premise that the methods used in \cite{M} allow to construct the 4-manifold in an explicit way. In fact, if $M$ is a dodecahedral manifold tessellated into $n$ dodecahedra, the result of \cite{M} yields, under certain hypotheses, a 4-manifold $N$ tessellated into at most $2^{44}\cdot n$ hyperbolic right-angled 120-cells {\cite[Proof of Theorem 3]{M}} in which $M$ geodesically embeds. Notice that this gives a bound on the volume of $N$. This bound on the number of 120-cells is not sharp and in practice our examples are tessellated in a lot less 120-cells.

Using this construction, the manifold $M$ is non-separating inside $N$. However, inside $N$ it is easy to find a certain number of copies of $M$ that, all together, separate. At this point if $M$ is an $L$-space one can use an argument as in \cite[Corollary~2.5]{AL} to obtain a separating $L$-space $M'$ that is diffeomorphic to the connected sum of several copies of $M$. There is also a natural way to ensure that $b_2^+ (N_i)\geq 1$. The details are discussed in Section \ref{sec:building_4_man}.

With the help of Theorem \ref{thm:classdodman}, we prove the following:

\begin{theorem}\label{cor:120cellmanifold}
There are two hyperbolic 4-manifolds $\mathcal{N}_{11}$ and $\mathcal{N}_{28}$ tessellated with $2^9$ right-angled 120-cells that can be obtained as $N_1\cup_{M'} N_2$, where the separating hypersurface $M'$ is an $L$-space and such that $b_2^+(N_i)\geq 1$ for $i=1,2$.
\end{theorem}

The manifolds in the statement are built by colouring 4-manifolds with right-angled corners tessellated in 120-cells (see Section \ref{sec:building_4_man}) and are explicitly described in Section \ref{sec:concrete_examples}. It is possible to generalise the notion of colouring lowering the number of 120-cells in the tessellation to $2^8$; we present this construction in Section \ref{sec:generalised_colouring}, but we do not go into the details.
The Betti numbers with coefficients in $\mathbb{R}$ and $\mathbb{Z}_2$ of the explicit examples that we build can be found in Tables \ref{table:homology_N_11} - \ref{table:homology_N_28} -  \ref{table:homology_M_11} - \ref{table:homology_M_28}. We also tried to obtain the integral homology, but the computation was too intensive for our computational resources.
We point out that dodecahedral manifolds satisfy the hypotheses of \cite[Theorem 1.1]{KRS} and therefore this theorem can be used to prove that they embed in hyperbolic 4-manifolds, but the use of the construction of \cite{M} allows us to describe the 4-manifolds explicitly.

The proof of Theorem \ref{thm:classdodman} is achieved by rigorous computer-assisted computations. In particular, we make use of the code written by Nathan Dunfield \cite{D} and SnapPy \cite{snappy} in a Sage \cite{sagemath} environment. All the code used is available at \cite{codice}, and it can be used to check if a given manifold is an $L$-space.

The proof of Theorem \ref{cor:120cellmanifold} is also computer-assisted. We make use of Regina \cite{regina} in a Sage \cite{sagemath} environment, and in particular of modules written by Tom Boothby, Nathann Cohen, Jeroen Demeyer, Jason Grout, Carlo Hamalainen, and William Stein. All the code used is available at \cite{codice}.
\begin{table}
\begin{center}
\begin{tabular}{c || c | c | c | c }
Index & N. dod. & Volume & $H_1$ & $L$-space \\
\hline
\rowcolor{lightgray} 0 & 2 & 8.612 & Z/11 + Z/11 & Yes \\
1 & 2 & 8.612 & Z/87 & No \\
\rowcolor{lightgray} 2 & 2 & 8.612 & Z/4 + Z/28 & Yes \\
3 & 2 & 8.612 & Z/2 + Z/2 + Z/2 + Z/2 & No \\
4 & 2 & 8.612 & Z/3 + Z/15 & No \\
5 & 4 & 17.225 & Z/2 + Z/4 + Z/180 & No \\
6 & 4 & 17.225 & Z/714 & No \\
7 & 4 & 17.225 & Z/2 + Z/2 + Z/44 & No \\
\rowcolor{lightgray} 8 & 4 & 17.225 & Z/7 + Z/182 & Yes \\
9 & 4 & 17.225 & Z/4 + Z/44 & No \\
10 & 4 & 17.225 & Z/2 + Z/4 + Z/60 & No \\
\rowcolor{lightgray} 11 & 4 & 17.225 & Z/2 + Z/2 + Z/120 & Yes \\
12 & 4 & 17.225 & Z/12 + Z/12 & No \\
13 & 4 & 17.225 & Z/2 + Z/2 + Z/144 & No \\
14 & 4 & 17.225 & Z/2 + Z/4 + Z/4 + Z/4 & No \\
\rowcolor{lightgray} 15 & 4 & 17.225 & Z/513 & Yes \\
16 & 4 & 17.225 & Z/4 + Z/4 + Z/8 & No \\
17 & 4 & 17.225 & Z/2 + Z/2 + Z/2 + Z/2 + Z/4 & No \\
18 & 4 & 17.225 & Z/4 + Z/4 + Z/8 & No \\
19 & 4 & 17.225 & Z/4 + Z/4 + Z/8 & No \\
20 & 4 & 17.225 & Z/2 + Z/4 + Z/8 & No \\
21 & 4 & 17.225 & Z/2 + Z/4 + Z/8 & No \\
22 & 4 & 17.225 & Z/4 + Z/4 + Z/8 & No \\
23 & 4 & 17.225 & Z/2 + Z/2 + Z/56 & No \\
24 & 4 & 17.225 & Z/4 + Z/4 + Z/8 & No \\
25 & 4 & 17.225 & Z/2 + Z/2 + Z/2 + Z/4 + Z/4 & No \\
26 & 4 & 17.225 & Z/2 + Z/2 + Z/8 + Z/8 & No \\
27 & 4 & 17.225 & Z/4 + Z/4 + Z/8 & No \\
\rowcolor{lightgray} 28 & 4 & 17.225 & Z/2 + Z/2 + Z/8 + Z/8 & Yes \\
        
\end{tabular}
\vspace{.2 cm}
\nota{The manifolds in $\mathscr{D}$: the hyperbolic 3-manifolds tessellated with four or less right-angled dodecahedra that are rational homology spheres. The indexing is the one provided by SnapPy, where $\mathscr{D}$ can be accessed by typing \texttt{CubicalOrientableClosedCensus(betti=0)}. We list the index of the manifold as given in the SnapPy census, the number of right-angled dodecahedra in its tessellation, its approximated volume, its first homology and whether it is an $L$-space or not. Compare with Table \ref{table:dodecahedralmanifoldswithorderable}.} 
\label{table:dodecahedralmanifolds}
\end{center}
\end{table}
\newline

\textbf{Structure of the paper.} Section \ref{sec:finding_L_spaces} is dedicated to the proof of Theorem \ref{thm:classdodman}. In Section \ref{section:L_space_intervals} we recall the definition of $L$-space and the main theorem of \cite{RR}. In Section \ref{sec:algorithm_finding_L_spaces}, we recall the techniques used by Dunfield in \cite{D} to classify the $L$-spaces in $\mathscr{Y}$. In Section \ref{section:proving_L_space}, we elaborate on these techniques and describe an algorithm that can be used to prove that a hyperbolic rational homology sphere is an $L$-space. In Section \ref{sec:classification_dodecahedral_manifolds}, we complete the proof of Theorem \ref{thm:classdodman}.

Section \ref{sec:building_4_man} contains the details of the construction necessary for the proof of Theorem \ref{cor:120cellmanifold}. In Section \ref{subsection:embedding dodecahedral} we recall the general theory of manifolds with right-angled corners and colourings; then, in Section \ref{sec:concrete_examples}, we move to the explicit construction.

Appendices \ref{appendix:alg}-\ref{appendix:example} contain a detailed description of the algorithm used in Section \ref{sec:algorithm_finding_L_spaces}, with some examples.
\\

\textbf{Acknowledgments.} We would like to thank Nathan Dunfield, Matthias Goerner, Francesco Lin, Alan Reid, Stefano Riolo, and Leone Slavich for the useful discussions. We also thank Nathan Dunfield for all the code he released under the public domain, without which this work would have been much harder (if not impossible). We also thank the referee for the useful comments.

\section{Finding $L$-spaces}\label{sec:finding_L_spaces}
We recall the definition of \emph{$L$-space}.
\begin{definition}\label{defn:Lspace}
A rational homology 3-sphere (\qhs{} for short) $M$ is an $L$-space if $\widehat{HF}(M)$ is a free abelian group whose rank coincides with $|H_1(M,\Z)|$.
\end{definition}
In this definition, $\widehat{HF}(M)$ denotes the hat-version of the Heegaard Floer homology of $M$. Heegaard Floer homology is a topological invariant of $3$-manifolds introduced by Ozsv\'ath and Szab\'o in \cite{OS3} that has found many important and profound applications in low dimensional topology (see \cite{J} and \cite{H} for some examples).
Examples of $L$-spaces are given by the $3$-sphere, lens spaces or more generally any elliptic space \cite[Proposition~2.3]{OS1}. 
When $M$ is a rational homology sphere, it is proved in  \cite[Proposition~5.1]{OS2} that $\rank( \widehat{HF}(M))\geq |H_1(M,\Z)|$ holds and therefore $L$-spaces have minimal Heegaard Floer homology. 
These spaces have been widely investigated recently and they are one of the objects of study of the following conjecture:
\begin{conj}[$L$-space conjecture]\label{conj:Lspaceconj}
For an irreducible orientable rational homology $3$-sphere $M$, the following are
equivalent:
\begin{enumerate}[label=\arabic*)]
    \item  $M$ supports a co-orientable taut foliation;
    \item  $M$ is not an $L$-space;
    \item  $\pi_1(M)$ is left-orderable.
\end{enumerate}
\end{conj}

For the definition of \emph{co-orientable taut foliation} we refer to \cite{Bow} and \cite{KR}; for the one of \emph{left-orderable} we refer to \cite{BGW}.

The equivalence between $1)$ and $2)$ was conjectured by Juh\'asz in \cite{J}, while the equivalence between $2)$ and $3)$ was conjectured by Boyer, Gordon and Watson in \cite{BGW}.

Concerning this conjecture, for the purposes of this paper (see Section \ref{section:proving_L_space})  we only point out the following theorem, which is a consequence of the works of Oszváth-Szábo \cite{OS}, Bowden \cite{Bow} and Kazez-Roberts \cite{KR}. 

\begin{theorem}\label{CTF implica NLS}
Let $M$ be a rational homology sphere that supports a co-orientable taut foliation. Then $M$ is not an $L$-space.
\end{theorem}

As the following proposition states, $L$-spaces can also be used to prove that a $4$-manifold has vanishing Seiberg-Witten invariants. For a proof we refer to \cite{KM} and also \cite[Proposition 2.2]{AL}.

\begin{prop}\label{separating L-space}
Let $N$ be a $4$-manifold given as $N=N_1\cup_{M} N_2$, where $M$ is an $L$-space and where $b_2^+(N_i) \geq 1$ for $i=1,2$. Then all the Seiberg-Witten invariants of
N vanish.
\end{prop}
To be precise, the right notion of $L$-space to use in the previous proposition is the one of \emph{monopole Floer} $L$-space, but the monopole Floer/Heegaard Floer correspondence, proved by the works of Kutluhan-Lee-Taubes \cite{KLT1, KLT2, KLT3, KLT4, KLT5}, implies that a manifold is a monopole Floer $L$-space if and only if it is an $L$-space in the sense of Definition \ref{defn:Lspace}. A different proof of the correspondence monopole Floer/Heegaard Floer can be obtained by combining the works of Taubes \cite{Tau1, Tau2, Tau3, Tau4, Tau5} and Colin-Ghiggini-Honda \cite{CGH}.

\subsection{$L$-space intervals}\label{section:L_space_intervals}

In this section we recall some definitions and the main result of \cite{RR}, that will be a key ingredient of our algorithm. 

Let $Y$ be a rational homology solid torus (\qht{} for short), \emph{i.e.}\ an oriented compact $3$-manifold with toroidal boundary and $H_{*}(Y,\Q)\cong H_{*}(\D^2\times S^1,\Q)$.
We define the set of \emph{slopes} in $Y$
$$
Sl(Y)=\{ \alpha\in H_1(\partial Y,\Z)\,| \,\alpha\text{ is primitive}\}/ \pm 1.
$$
We are interested in the Dehn fillings of $Y$ and it is a well-known fact (see for example \cite[Section 10.1.1]{Martellibook}) that each Dehn filling of $Y$ is determined by a slope $[\alpha]\in Sl(Y)$ and conversely that each slope $[\alpha]\in Sl(Y)$ determines a Dehn filling of $Y$, that we will denote with $Y(\alpha)$.
Since $Y$ is a \qht{}, there exists a distinguished slope called \emph{homological longitude}, that is defined as $[l]$, where $l\in H_1(\partial Y,\Z)$ is a generator of the kernel of the map
$$
\iota:H_1(\partial Y,\Z)\rightarrow H_1(Y,\Z)
$$
induced by the inclusion $i:\partial Y\hookrightarrow Y$; since $l$ is unique up to sign, the class $[l]\in Sl(Y)$ is well-defined.
For example, when $Y$ is the exterior of a knot $K$ in $S^3$, the homological longitude $[l]$ coincides with the slope defined by the (canonical) longitude of $K$.

Since we want to study the Dehn fillings of $Y$ that are $L$-spaces we define the set of \emph{$L$-spaces filling slopes} of $Y$
$$
L(Y)=\{[\alpha]\in Sl(Y)\,|\, Y(\alpha) \text{ is an $L$-space}\}
$$
and we say that $Y$ is \emph{Floer simple} if there exist at least two Dehn fillings of $Y$ that are $L$-spaces, \emph{i.e.}\ if $|L(Y)|>1$.

It turns out that if $Y$ is Floer simple, then it is possible to describe explicitly the set of $L$-spaces filling slopes of $Y$. For example, if $Y$ is the exterior of a knot $K$ in $S^3$ we can fix $(\mu_K,\lambda_K)$ the usual meridian-longitude basis of $H_1(\partial Y,\Z)$ and we have a bijection 

\begin{align*}
&Sl(Y)&\longleftrightarrow& &\Q&\cup \{\infty\}\\
[\alpha]&=[p\mu_K+q\lambda_K]&\longleftrightarrow& &&\frac{p}{q}. 
\end{align*}
Observe that in this case $Y$ is Floer simple if and only if it has a non-trivial $L$-space filling, since $\{\infty\}$ yields $S^3$ that is an $L$-space. Moreover, by considering the mirror of $K$ if necessary, we can suppose that if $Y$ has a non-trivial $L$-space filling slope then it is a positive one and in this case it has been proved in \cite{KMOS} that
$$
L(Y)=[2g(K)-1, \infty]\cap \left(\Q\cup \{\infty\}\right),
$$ 
where $g(K)$ denotes the genus of $K$.

When $Y$ is a more general \qht{} the structure of the set $L(Y)$ can be computed by knowing the \emph{Turaev torsion} of $Y$.
We only recall some properties of the Turaev torsion and refer the reader to \cite{T} for the precise definitions.

Fix an identification $H_1(Y,\Z)=\Z \oplus T$ where $T$ is the torsion subgroup and denote with $\phi:H_1(Y,\Z)\rightarrow \Z$ the projection induced by this identification. 
The Turaev torsion of $Y$ is an invariant of $Y$ that can be normalised so to be written as a formal sum
$$
\tau(Y)=\sum_{\substack{h\in H_1(Y,\Z) \\ \phi(h)\geq 0}}a_h h
$$
where $a_h$ is an integer for all $h\in H_1(Y,\Z)$, $a_0\ne 0$ and $a_h=1$ for all but finitely many $h\in H_1(Y,\Z)$. This invariant generalises the Alexander polynomial of $Y$ and in fact (see \cite[Section II.5]{T}) when $H_1(Y,\Z)$ has no torsion it holds
$$
\tau(Y)=\frac{\Delta(Y)}{1-t}\in \Z[[t]]
$$
where $(1-t)^{-1}$ is expanded as an infinite sum in positive powers of $t$ and the Alexander polynomial $\Delta(Y)$ is normalised so that $\Delta(Y)\in \Z[t]$, $\Delta(Y)(0)\ne 0$ and $\Delta(Y)(1)=1$. Notice that in this case the coefficients of $\tau(Y)$ are eventually equal to $1$ since they are eventually equal to the sum of all the coefficients of $\Delta(Y)$, and this value is exactly $\Delta(Y)(1)$.

Given $\tau(Y)$  we define the following subset of $H_1(Y, \Z)$
$$
D^{\tau}_{>0}=\{x-y\,| x\notin S[\tau(Y)],\, y\in S[\tau(Y)] \text{ and } \phi(x)>\phi(y)\}\cap \iota(H_1(\partial Y, \Z))
$$
where $S[\tau(Y)]=\{ h\in H_1(Y,\Z)\,|\, a_h\ne 0 \}$ is the \emph{support} of $\tau(Y)$, and where \[\iota:H_1(\partial Y,\Z)\rightarrow H_1(Y,\Z)\] is induced from the inclusion $\partial Y\hookrightarrow Y$.

We are now ready to state the main result of \cite{RR}. 

\begin{theorem}[{\cite[Theorem 1.6]{RR}}] \label{theorem RR}
If $Y$ is Floer simple, then either
\begin{itemize}
    \item $D_{>0}^\tau(Y)=\emptyset$ and $L(Y)=Sl(Y)\setminus [l]$, or
    \item $D_{>0}^\tau(Y)\ne\emptyset$ and $L(Y)$ is a closed interval whose endpoints are consecutive elements in $\iota^{-1}(D_{>0}^\tau(Y))$.
\end{itemize}
\end{theorem}
Recall that $[l]\in Sl(Y)$ denotes the homological longitude of $Y$, and since the corresponding Dehn filling of $Y$ is not a \qhs{}, by the definition of $L$-space it is obvious that $[l]\notin L(Y)$. We explain more precisely the second part of the statement of the theorem. Once we fix a basis $(\mu,\lambda)$ for $H_1(\partial Y,\Z)$, exactly as in the case of the exterior of a knot, we can define a map 

\begin{align*}
&H_1(\partial Y,\Z)&\longrightarrow& &\Q&\cup \{\infty\}\\
&p\mu+q\lambda&\longrightarrow& & &\frac{p}{q} 
\end{align*}
and this association induces an identification between $Sl(Y)$ and $\Q \cup \{\infty\}$. If the set $D^{\tau}_{>0}$ is not empty, we can apply this map to the set $\iota^{-1}(D^{\tau}_{>0})$ and Theorem \ref{theorem RR} states that if $Y$ is Floer simple then $L(Y)$ is a closed interval in $Sl(Y)=\Q\cup \{\infty\}$ whose endpoints are consecutive elements in the image of $\iota^{-1}(D^{\tau}_{>0})$ in $\Q \cup \{\infty\}$.\label{L-spaces_interval}

\subsection{An algorithm for finding $L$-spaces}\label{sec:algorithm_finding_L_spaces}
In this section we describe an algorithm that can determine whether a hyperbolic rational homology sphere is an $L$-space. The two main ingredients of this algorithm are the Rasmussen-Rasmussen Theorem \ref{theorem RR} and part of the work of Dunfield in \cite{D}. 

In \cite{D} Dunfield considered a census $\mathscr{Y}$ of more than $300,000$ hyperbolic rational homology spheres. These are obtained as Dehn fillings of $1$-cusped hyperbolic $3$-manifolds that can be triangulated with at most $9$ ideal tetrahedra; the latter were enumerated by Burton \cite{B}. We denote the census of these \qht{}s with $\mathscr{C}$. The hyperbolic structure of a \qht{} is always intended on its interior; notice that such a metric is always finite-volume, see \cite[Proposition D.3.18]{BenedettiPetronio}.

Dunfield, while investigating the $L$-space conjecture for the manifolds in $\mathscr{Y}$, has in particular proved the following:

\begin{theorem}[{\cite[Theorem 1.6]{D}}]\label{theorem Dunfield}
Of the $307,301$ hyperbolic rational homology $3$-spheres in $\mathscr{Y}$ exactly $144,298$ $(47.0 \%)$ are $L$-spaces and $163,003$ $(53.0\%)$ are non-L-spaces.
\end{theorem}

The key idea that Dunfield used to prove Theorem \ref{theorem Dunfield} is to use Theorem \ref{theorem RR} to start a bootstrapping procedure. In fact by virtue of the Rasmussen-Rasmussen theorem, it is sufficient to know two $L$-spaces fillings of a rational homology solid torus in $\mathscr{C}$ to determine exactly its set of $L$-space filling slopes, and therefore to obtain information about the $L$-space status of the manifolds in $\mathscr{Y}$. This, together with the fact that many of the manifolds in $\mathscr{Y}$ admit multiple descriptions as fillings of manifolds in $\mathscr{C}$, allows to increase simultaneously the level of knowledge about the manifolds in $\mathscr{Y}$ and in $\mathscr{C}$.

Notice that the manifolds in $\mathscr{Y}$ have bounded volume: since they are fillings of $1$-cusped manifolds that are triangulated with at most $9$ ideal tetrahedra, their volume is bounded by $9v_3 < 9.14$, where $v_3$ is the maximal volume of a hyperbolic tetrahedron.

\subsection{Proving that a manifold is an $L$-space}\label{section:proving_L_space}
Following the bootstrap idea of Dunfield, we now present an algorithm that can be used to determine in certain cases whether a hyperbolic rational homology sphere is an $L$-space. Before describing the algorithm, we recall a definition:
\begin{definition}
A rational homology solid torus $Y$ is \emph{Turaev simple} when every coefficient of the Turaev torsion of $Y$ is either $0$ or $1$.
\end{definition}
In \cite[Proposition 1.4]{RR} it is proved that being Turaev simple is a necessary condition for being Floer simple. The converse is not true in general.

We also use the following notation: when $M$ is a \qhs{}, we say that the \emph{L-space
value of $M$} is \texttt{True} if $M$ is an $L$-space and \texttt{False} otherwise.

The algorithm takes in input a rational homology sphere $M$ and returns  \texttt{True} if a proof that $M$ is an $L$-space is found and \texttt{False} otherwise.

The rough idea of the algorithm is the following: 

\begin{enumerate}
    \item We start with $M$, a hyperbolic \qhs{}. If it belongs to the Dunfield census, we return its $L$-space value; otherwise we go on with Step \ref{item:FindingDrilling}.;

    \item \label{item:FindingDrilling} we drill the shortest geodesic given by SnapPy out of $M$ so to obtain $Y$, a hyperbolic \qht{} that is Turaev simple. We fix a meridian-longitude basis of $\partial Y$ so that the filling $\sfrac{1}{0}$ on $Y$ gives back $M$, and we identify $Sl(Y)$ with $\mathbb{Q} \cup \{\infty\}$;
    
    \item \label{item:ComputeInterval} using a script provided by Dunfield in \cite{D}, we compute $\iota^{-1}(D_{>0}^\tau(Y))$ and we select an interval $\mathcal{I}$ in $Sl(Y)$ whose endpoints are consecutive elements in $\iota^{-1}(D_{>0}^\tau(Y))$ and such that $\sfrac{1}{0}\in \mathcal{I}$. In the case $D_{>0}^\tau(Y)=\emptyset$ we take $\mathcal{I}$ as $Sl(Y)\setminus{[l]}$, where $[l]$ is the homological longitude of $Y$;
    
    \item we search for two slopes in $\mathcal{I}$ so that the associated fillings are hyperbolic and have minimal volumes. We denote these fillings with $M_1$ and $M_2$. By Theorem \ref{theorem RR}, if we prove that they are $L$-spaces, then $M$ also is;
    
    \item we start two new instances of the algorithm with $M_1$ and $M_2$; if they both return \texttt{True}, we return \texttt{True}; otherwise, we return \texttt{False}.
    
\end{enumerate}

The details of the algorithm and one example can be found in Appendices \ref{appendix:alg}-\ref{appendix:example}. For the moment, let us underline the following facts:

\begin{itemize}

    \item the answer \texttt{True} is rigorous, but we stress that the answer \texttt{False} does not imply that $M$ is not an $L$-space. In particular, the rational homology torus $Y$ can be not Floer simple. We simply assure that it is Turaev simple, since this is an easier condition to check. This means that, even if we start with an actual $L$-space, the algorithm can return \texttt{False}: $M$ could be the only filling of $Y$ that is an $L$-space;
    
    \item we need to avoid that the algorithm enters an infinite loop: for example, it can happen that $M$ is the lowest-volume filling of $Y$ inside $\mathcal{I}$. This would cause an infinite loop. In the code there is a check to avoid such problems;
    
    \item we do not know whether a rational homology torus $Y$ as in Step \ref{item:FindingDrilling}.\ always exists, even if we assume that $M$ is an $L$-space. See Questions \ref{question:floer_simple} - \ref{question:floer_simple_geodesic};

    \item in Step \ref{item:ComputeInterval}.\ there could be two of such intervals; this happens when $\sfrac{1}{0}$ is an element of $\iota^{-1}(D_{>0}^\tau(Y))$. In this case, we look for $L$-spaces in both such intervals;
    
    \item there is no guarantee that the volumes of $M_1$ and $M_2$ are smaller than the volume of $M$. In particular, the algorithm is not guaranteed to end;
    
    \item we need the hyperbolicity of $M$ for two reasons: SnapPy can drill curves only when a manifold has a hyperbolic structure, and we use the hyperbolic volume;

    \item the key idea of the algorithm is to minimize the hyperbolic volume of $M_1$ and $M_2$ to approach the Dunfield census. While $M$ having small volume does not guarantee that $M$ is in $\mathscr{Y}$, in practice this condition makes the algorithm terminate quite fast. 
    
\end{itemize}

Despite these potential problems, the algorithm was powerful enough to prove the following:

\begin{prop}\label{prop:areLspaces}
The manifolds with index 0, 2, 8, 11, 15 and 28 in the census $\mathscr{D}$ (recall Definition \ref{defn:censusD}) are $L$-spaces (see Tables \ref{table:dodecahedralmanifolds} - \ref{table:dodecahedralmanifoldswithorderable}).\hfill{$\square$}
\end{prop}

\subsection{Classification of Dodecahedral manifolds}\label{sec:classification_dodecahedral_manifolds}

We are left with proving that the manifolds in Proposition \ref{prop:areLspaces} are the only ones in $\mathscr{D}$ that are $L$-spaces. By virtue of Theorem \ref{CTF implica NLS}, if $M$ supports a co-orientable taut foliation then it is not an $L$-space. In \cite[Section 7]{D}, Dunfield introduced the notion of \emph{foliar orientation} and used it to construct co-orientable taut foliations on a given \qhs{}. He also provided an algorithm that searches for foliar orientations. Applying his algorithm, we are able to prove the following:

\begin{lemma}
The manifolds with index different from 0, 2, 8, 11, 15 and 28 in the census $\mathscr{D}$ are not $L$-spaces (see Tables \ref{table:dodecahedralmanifolds} - \ref{table:dodecahedralmanifoldswithorderable}).\hfill{$\square$}
\end{lemma}

This lemma, together with Proposition \ref{prop:areLspaces}, implies the following:

\begin{namedtheorem}[\ref{thm:classdodman}]
Of the 29 manifolds in $\mathscr{D}$, 6 are $L$-spaces and 23 are not; see Tables \ref{table:dodecahedralmanifolds} - \ref{table:dodecahedralmanifoldswithorderable}.
\end{namedtheorem}

In \cite{D}, Dunfield provides several algorithms to check the left-orderability and the non-left-orderability of the fundamental group of a \qhs{}. We apply these algorithms, and the details can be found in Table \ref{table:dodecahedralmanifoldswithorderable}. In particular, all the results are consistent with Conjecture \ref{conj:Lspaceconj}.

\begin{table}
\begin{center}
\begin{tabular}{c || c | c | c | c | c}
Index & N. dod. & $L$-space & Co-or.\ taut foliation & Left-orderable $\pi_1$ \\
\hline
0 & 2 & Yes$^1$ & No$^2$ & No$^5$ \\
1 & 2 & No$^2$ & Yes$^3$ & ? \\
2 & 2 & Yes$^1$ & No$^2$ & No$^5$ \\
3 & 2 & No$^2$ & Yes$^3$ & Yes$^4$ \\
4 & 2 & No$^2$ & Yes$^3$ & Yes$^4$ \\
5 & 4 & No$^2$ & Yes$^3$ & ? \\
6 & 4 & No$^2$ & Yes$^3$ & ? \\
7 & 4 & No$^2$ & Yes$^3$ & Yes$^4$ \\
8 & 4 & Yes$^1$ & No$^2$ & No$^5$ \\
9 & 4 & No$^2$ & Yes$^3$ & Yes$^4$ \\
10 & 4 & No$^2$ & Yes$^3$ & ? \\
11 & 4 & Yes$^1$ & No$^2$ & No$^5$ \\
12 & 4 & No$^2$ & Yes$^3$ & Yes$^4$ \\
13 & 4 & No$^2$ & Yes$^3$ & ? \\
14 & 4 & No$^2$ & Yes$^3$ & Yes$^4$ \\
15 & 4 & Yes$^1$ & No$^2$ & No$^5$ \\
16 & 4 & No$^2$ & Yes$^3$ & ? \\
17 & 4 & No$^2$ & Yes$^3$ & Yes$^4$ \\
18 & 4 & No$^2$ & Yes$^3$ & Yes$^4$ \\
19 & 4 & No$^2$ & Yes$^3$ & Yes$^4$ \\
20 & 4 & No$^2$ & Yes$^3$ & ? \\
21 & 4 & No$^2$ & Yes$^3$ & Yes$^4$ \\
22 & 4 & No$^2$ & Yes$^3$ & ? \\
23 & 4 & No$^2$ & Yes$^3$ & Yes$^4$ \\
24 & 4 & No$^2$ & Yes$^3$ & Yes$^4$ \\
25 & 4 & No$^2$ & Yes$^3$ & Yes$^4$ \\
26 & 4 & No$^2$ & Yes$^3$ & Yes$^4$ \\
27 & 4 & No$^2$ & Yes$^3$ & Yes$^4$ \\
28 & 4 & Yes$^1$ & No$^2$ & No$^5$ \\

\end{tabular}
\vspace{.2 cm}
\nota{The manifolds in $\mathscr{D}$: the hyperbolic 3-manifolds tessellated with four or less right-angled dodecahedra that are rational homology spheres. The indexing is the one provided by SnapPy, where $\mathscr{D}$ can be accessed by typing \texttt{CubicalOrientableClosedCensus(betti=0)}. The apices indicate in which way we obtain the result. In particular: 1)~algorithm in Section~\ref{section:proving_L_space}; 2)~Theorem \ref{CTF implica NLS} by \cite{OS, Bow, KR}; 3)~algorithm in \cite{D}, see \cite[Section 7]{D}; 4)~algorithm in \cite{D}, see \cite[Section 9]{D}; 5)~algorithm in \cite{D}, see \cite[Section 5]{D}. Compare with Table \ref{table:dodecahedralmanifolds}.} 
\label{table:dodecahedralmanifoldswithorderable}
\end{center}
\end{table}

\section{Building up the 4-manifold}\label{sec:building_4_man}

We now show how the existence of a dodecahedral manifold that is an $L$-space can be used to prove Theorem \ref{cor:120cellmanifold}. To do this, we recall some concepts behind the construction in \cite{M}.

\subsection{Embedding dodecahedral manifolds in 4-manifolds with corners}\label{subsection:embedding dodecahedral}

Let $\mathbb{D}^n$ be the disc model for the hyperbolic space and $\mathcal{O}^n=\{x\in \mathbb{D}^n\mid x_i\ge0 \}$ the positive orthant. A \textit{hyperbolic $n$-manifold with (right-angled) corners} $W$ is a topological $n$-manifold equipped with an atlas taking values in open subsets of $\mathcal{O}^n$ and transition maps that are restrictions of isometries. One can visualize this as the natural extension of a manifold with geodesic boundary, where the atlas takes values in open subsets of $\{x\in \mathbb{D}^n\mid x_1\ge0 \}$.

\begin{figure}[H]
    \centering
    \includegraphics[scale=0.3]{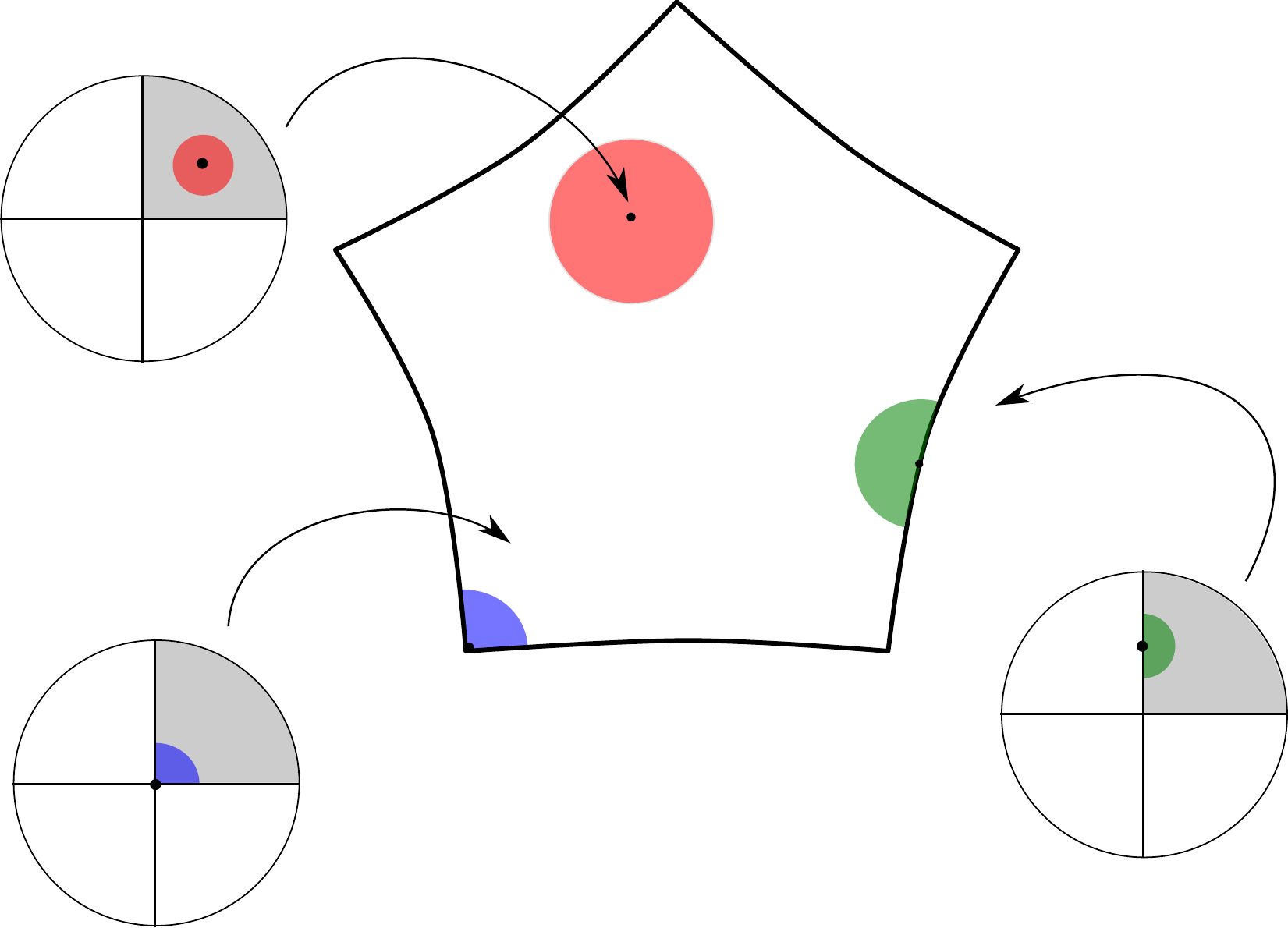}
    \nota{An example of charts for a 2-manifold with right-angled corners.}
    \label{im:orbifold example}
\end{figure}

The boundary $\partial W$ is the set of points in $W$ that do not admit a neighborhood homeomorphic to $\mathbb{R}^n$; it is stratified into vertices, edges, ..., $s$-faces,..., and facets. Each $s$-face is an $s$-manifold with corners, and distinct $s$-faces meet at right-angles. A facet is called \textit{isolated} if it does not meet any other facet, and as such it must be a geodesic boundary component of $W$. 

Under certain hypotheses, gluing manifolds with corners along (possibly more than one) pair of isometric facets yields another manifold with corners, as the following example shows.

\begin{example}\label{ex:corner-angle manifold example}
\normalfont
Consider the surface $S$ in Figure~\ref{im:corner-angled}-right, obtained by gluing two copies of the right-angled pentagon on the left along the coloured edges. This is a hyperbolic manifold with corners, with edges $F_1$, $F_2$ and $F_3$ and vertices $p$ and $q$. We have $F_1\cap F_2 =\{p,q\}$, while the facet $F_3\cong S^1$ is isolated.

\begin{figure}[H]
\centering
  \centering
  \includegraphics[scale=0.5]{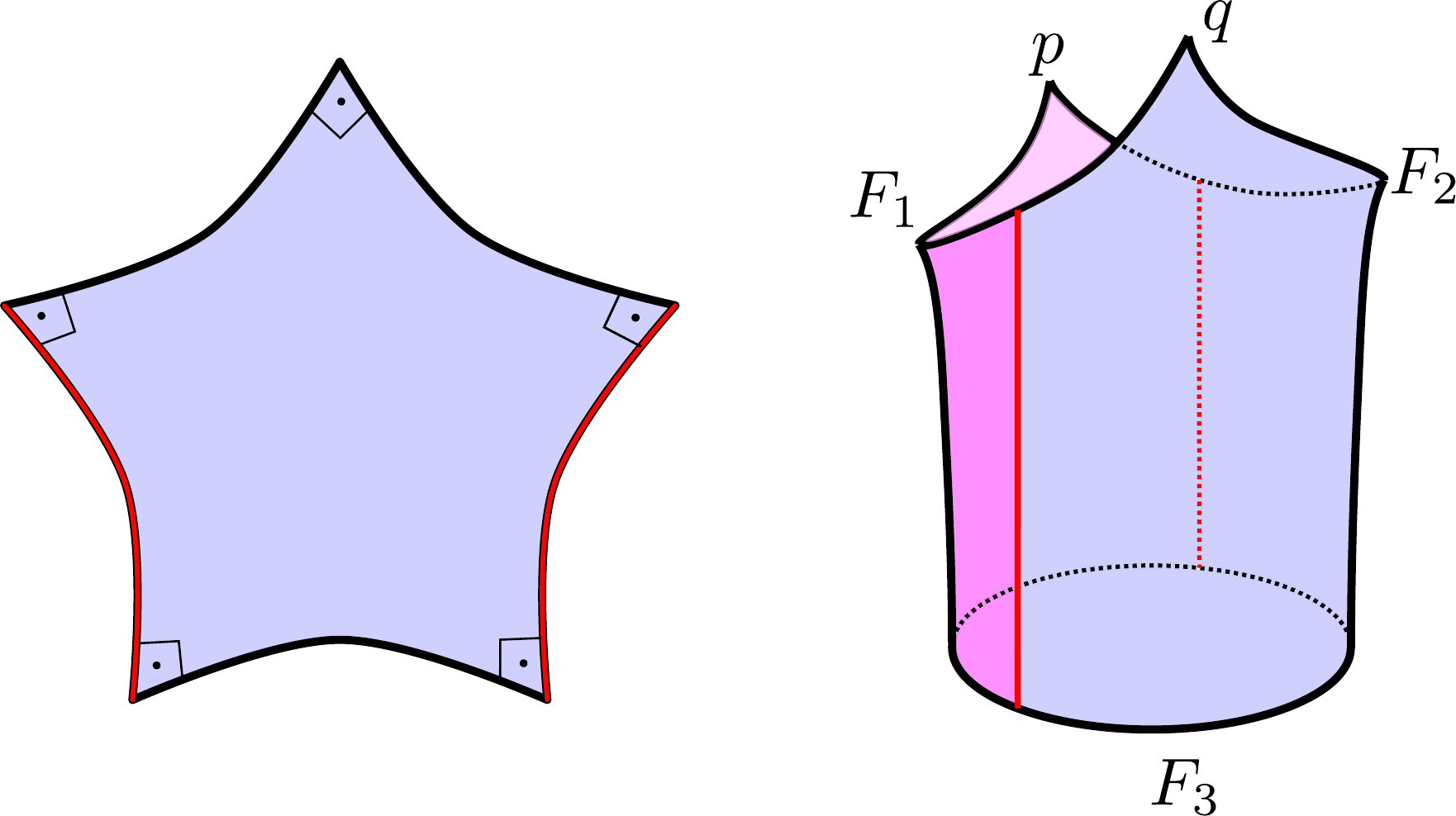}
\nota{An example of a right-angled pentagon and the manifold with corners obtained by mirroring it along the red edges.}
\label{im:corner-angled}
\end{figure}
\end{example}

More generally, similarly to right-angled polytopes \cite[Section 2]{KMT}, a manifold with corners can be coloured along its facets. More precisely, let $\mathcal{F}_W$ be the set of facets of a manifold with corners $W$. 
A \textit{$k$-colouring} of $W$ is a surjective map $\lambda:\mathcal{F}_W\to \{1,...,k\}$ that associates distinct numbers (called \textit{colours}) to adjacent facets. Let $e_1, \ldots, e_k$ be the canonical basis of $\mathbb{Z}_2 ^k$. We can define the topological space $\mathcal{M}_\lambda=(W\times \mathbb{Z}_2 ^k)/_\sim$, where distinct $W\times \{u\}$, $W\times \{v\}$ are glued along the identity on a facet $F\in \mathcal{F}_W$ if $u-v=e_{\lambda(F)}$.

\begin{prop}[\cite{M}, Proposition 6]\label{prop:corner-angle colouring}
The resulting $\mathcal{M}_\lambda$ is a connected, orientable, hyperbolic $n$-manifold tessellated into $2^k$ copies of $W$. If $W$ is compact $\mathcal{M}_\lambda$ is closed. 
\end{prop}

\begin{example}\label{ex:corner-angle colouring}
\normalfont
Take the manifold with corners from Example \ref{ex:corner-angle manifold example} and its colouring $\lambda$ given by $\lambda(F_i)=i$. These colours are represented by red, blue and purple respectively in Figure \ref{im:colourings}.

An useful method to visualize the manifold $\mathcal{M}_\lambda$ is by iteratively mirroring along facets with the same colour. This procedure goes as follows:

\begin{itemize}
    \item we start with a manifold $M$ with corners and a $k$-colouring $\lambda$;
    \item we mirror $M$ along the facets $F$ such that $\lambda(F)=k$. We obtain a new manifold $M'$ with corners. The facets of $M'$ are of two types:
    \begin{itemize}
        \item the facets in $M$ that are adjacent to a facet with colour $k$ are mirrored along their intersection with these facets. Notice that, by using the combinatorics of the intersections of faces of $\mathcal{O}^n$, one can prove that the intersection of two facets is empty or a common sub-facet (see also \cite[Definition 2.2]{FKS} and the discussion thereafter). As an example, see what happens to the blue facet from Figure \ref{im:colourings}, $(a)$ to Figure \ref{im:colourings}, $(b)$;
        \item the facets in $M$ that do not intersect any facet with colour $k$ are doubled: each of them produces two isometric copies of itself. As an example, see what happens to the purple facet from Figure \ref{im:colourings}, $(a)$ to Figure \ref{im:colourings}, $(b)$;
    \end{itemize}
    \item we build a natural colouring $\lambda'$ on $M'$: each facet of $M'$ comes from a facet of $M$, by mirroring or doubling. We assign to each facet the colour of the corresponding facet in $M$. This is a $(k-1)$-colouring;
    \item if $k-1=0$, we have no facets lefts and we have obtained $\mathcal{M}_\lambda$; otherwise we start again this procedure with $M'$ and $\lambda'$.
\end{itemize}

By performing this on our example, after mirroring along the red and blue facets we get a compact manifold with $2^{3-1}=4$ copies of $F_3\cong S^1$ as boundary (see Figure \ref{im:colourings}, $(c)$). Then, by mirroring this manifold along its boundary, we get $\mathcal{M}_\lambda$ and the $4$ copies of $F_3$ as a separating submanifold.
\end{example}
\begin{figure}[H]
\centering
  \includegraphics[scale=0.17]{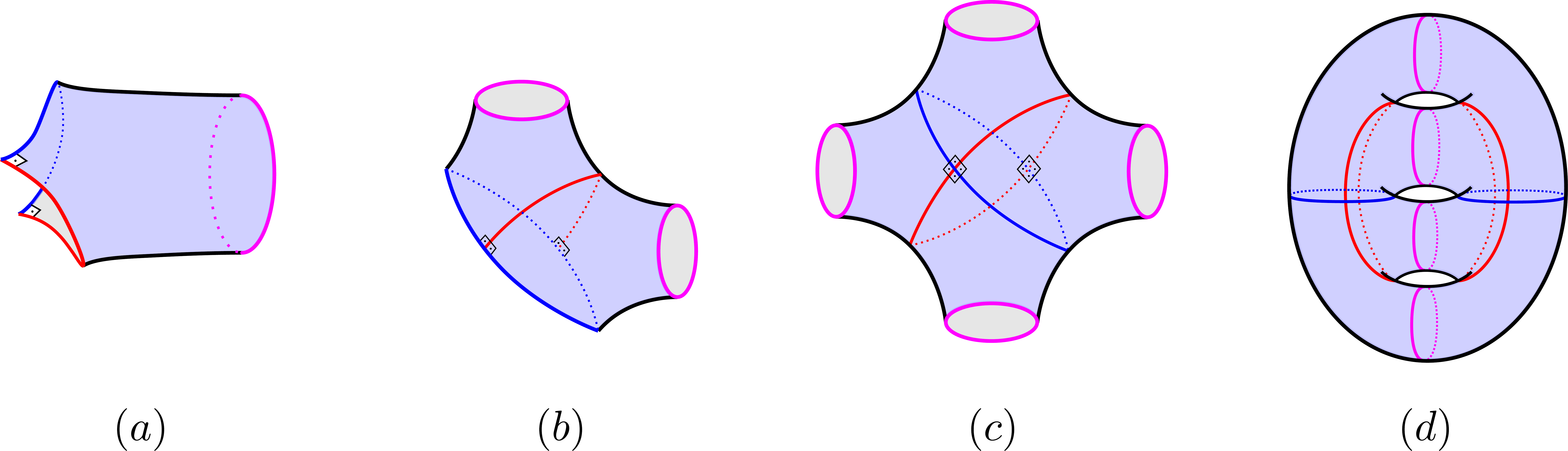}
\nota{From left to right, the four steps described in Example \ref{ex:corner-angle colouring}.}
\label{im:colourings}
\end{figure}

\begin{remark}
We point out that there can exist manifolds with corners that do not admit any colouring. In fact it can happen that some facet $F$ of a manifold with corners $W$ is not \emph{embedded}. In this case, such a facet is technically adjacent to itself and a colouring of $W$ should assign two different colours to $F$, which is impossible. This is an important issue we will take care of. See also Example \ref{example_non_embedded}.
\end{remark}

Concerning dodecahedral manifolds, \cite[Proposition 4 - Remark 5]{M} gives us the following:

\begin{prop}\label{prop:dodecahedral-embed}
Every dodecahedral manifold $M$ embeds geodesically in the interior of a connected, complete, compact, orientable hyperbolic $4$-manifold $W$ with corners. If $M$ is connected and tesselated into $k$ dodecahedra, $W$ is tesselated into $2k$ $120$-cells.
\end{prop}

\begin{remark}\label{rem:construction}
Notice that the 120-cell is a compact hyperbolic 4-dimensional polytope that has 120 3-facets that are right-angled dodecahedra. For more information on this, see \cite[Section 1]{MartelliSurvey}.
The idea of the construction in \cite{M} is to take one 120-cell $\mathcal{H}_i$ for each right-angled dodecahedron $\mathcal{D}_i$ of the decomposition of $M$, then consider each $\mathcal{D}_i$ as one facet of $\mathcal{H}_i$, and then extend each gluing between two faces $F\in \mathcal{D}_i$ and $G \in \mathcal{D}_j$ to a gluing of the dodecahedral facets $\mathcal{A} \in \mathcal{H}_i$ and $\mathcal{B} \in \mathcal{H}_j$  such that $\mathcal{D}_i \cap \mathcal{A}=F$ and $\mathcal{D}_j \cap \mathcal{B}=G$.
This operation produces a manifold with corners $W'$ in which $M$ is an isolated facet. The manifold $W$ is obtained by considering the mirror of $W'$ along $M$. This can be visualized by mirroring the manifold with corners in Figure \ref{im:corner-angled} along $F_3$. We denote these two copies of $W'$ inside $W$ as $W'_+$ and $W'_-$. As a consequence of this construction, we have that 
$$
\mathcal{F}_{W}=(\mathcal{F}_{W'_+}\setminus{M} )\sqcup (\mathcal{F}_{W'_-}\setminus{M}),
$$
and a facet in $\mathcal{F}_{W'_+}\setminus{M}$ is never adjacent to a facet in $\mathcal{F}_{W'_-}\setminus{M}$.
This decomposition induces a natural involution \[s \colon \mathcal{F}_{W} \to \mathcal{F}_{W}, \] that sends each facet $F$ of $W$ to the corresponding facet in the other copy of $W'$ in $W$. 
\end{remark}

We now describe a situation where a construction analogous to the one contained in the proof of Proposition \ref{prop:dodecahedral-embed} yields a manifold with corners with some non embedded facets. For the cases of our interests we tackle this issue in Proposition \ref{prop:tackle}.

\begin{example}\label{example_non_embedded}
\normalfont Recall that the regular right-angled hyperbolic hexagon is one facet of the Löbell polyhedron $R(6)$. This is a 3-dimensional right-angled hyperbolic polyhedron with $14$ faces: $2$ hexagonal faces and $12$ pentagonal faces arranged in the same pattern as the lateral surface of a dodecahedron. Löbell polyhedra were defined in \cite{V1}, see also \cite{V2} and see Figure \ref{R6} for a picture of $R(6)$.

\begin{figure}[H]

  \centering
  \includegraphics[scale=0.7]{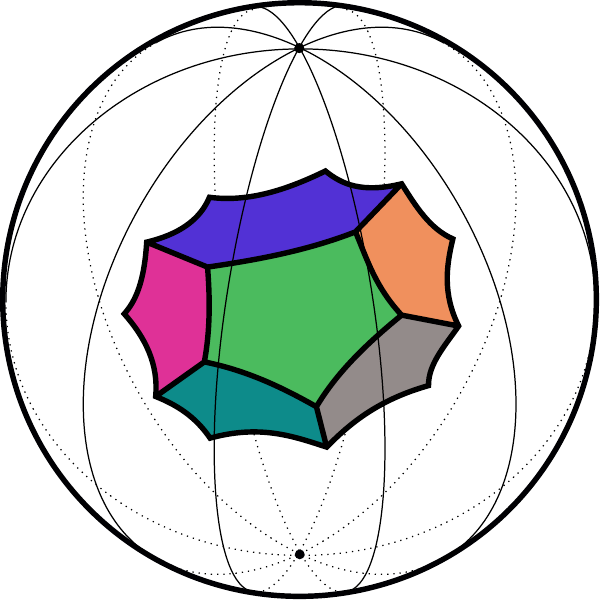}
\nota{The Löbell polyhedron $R(6)$ in the disc model of the hyperbolic space. In blue, one hexagonal facet. Opposite to it, there is another hexagonal facet, that is not visible. All the other facets are pentagons.}
\label{R6}
\end{figure}

Therefore if a closed surface $S$ is tessellated by regular right-angled hyperbolic hexagons, by performing the same construction of Proposition \ref{prop:dodecahedral-embed} it is possible to embed geodesically $S$ in the interior of a hyperbolic $3$-manifold with corners tessellated by copies of $R(6)$.
We now consider a regular right-angled hyperbolic hexagon $E$ and glue two of its edges as depicted in Figure \ref{hexagon-R(6)} $a)$. In this way we obtain a 2-dimensional manifold with corners and by colouring its three facets we obtain a closed hyperbolic surface $S$ tessellated by hexagons. When extending the gluings among the hexagons to the faces of the $R(6)$ polyhedra we have that two adjacent pentagonal faces (the coloured ones in Figure \ref{hexagon-R(6)} $b)$) of the polyhedron placed above $E$ are glued along one edge and become a non-embedded facet of the final 3-manifold with corners in which $S$ embeds.
\end{example}

\begin{figure}[H]

  \centering
  \includegraphics[scale=0.5]{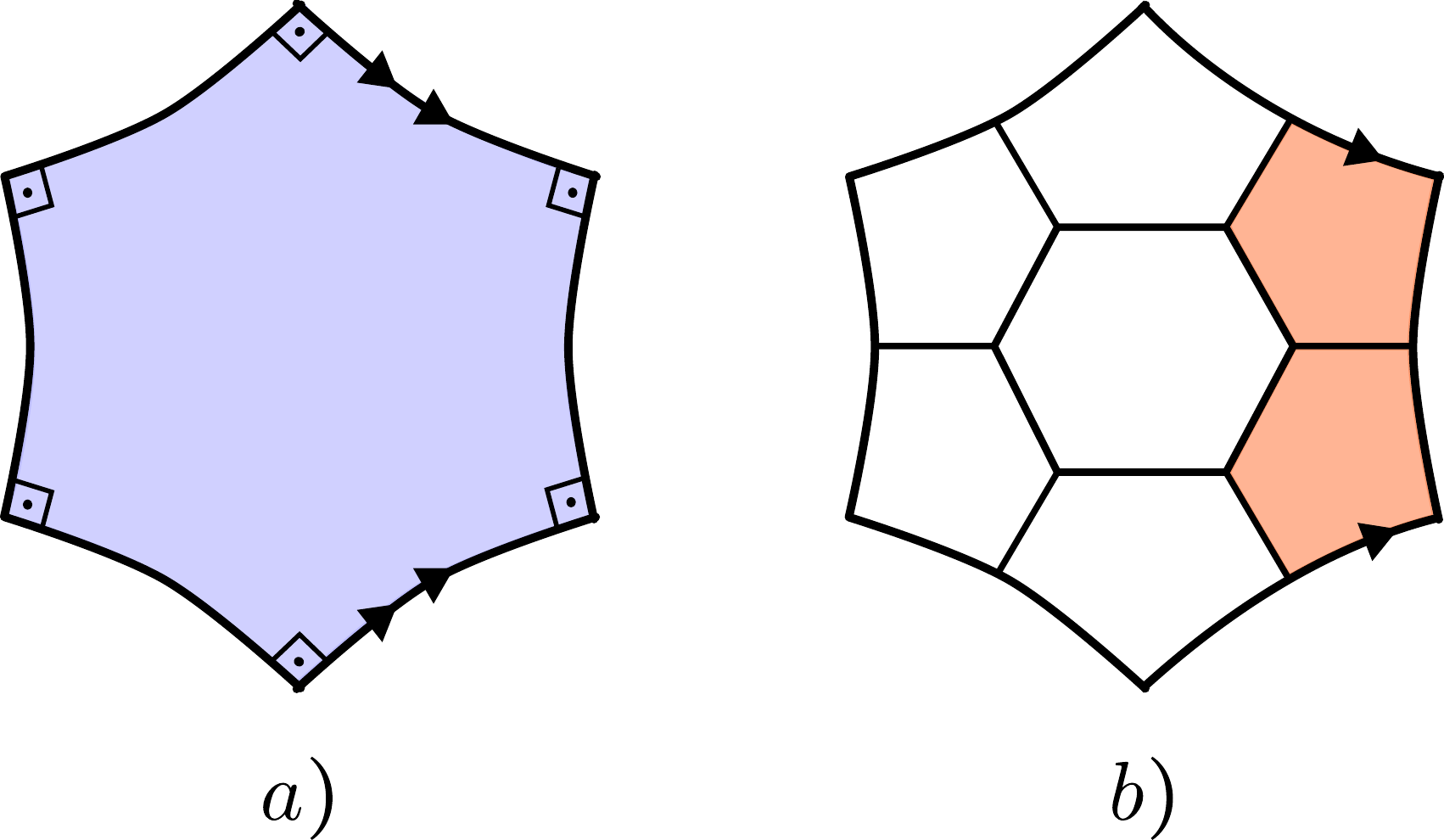}
\nota{On the left, the gluing that give raise to the non-embedded facet in Example \ref{example_non_embedded}. On the right, the facets of $R(6)$ not adjacent to the one that is identified with the hexagon; the two pentagons that become the same facet after the gluing are coloured.}
\label{hexagon-R(6)}
\end{figure}

We now introduce a special, yet very natural, type of colourings that will be useful for our constructions. Recall from Remark \ref{rem:construction} that there is an involution $s$ on the set of facets of the manifold $W$ provided by Proposition \ref{prop:dodecahedral-embed}.

\begin{definition}\label{defn:symm_col}
A colouring $\lambda$ of $W$ is \emph{symmetric} if $\lambda \circ s=\lambda$ (see Figure \ref{im:symm_col}).
\end{definition}

In other words, a colouring of $W$ is symmetric if and only if corresponding facets in the two copies of $W'$ in $W$ have the same colour. 
Asking for a symmetric colouring is not a restrictive request; in fact we have the following:

\begin{lemma}\label{lem:symm_colouring}
If the manifold with corners $W$ admits a $k$-colouring then it also admits a symmetric $h$-colouring with $h\leq k$.
\end{lemma}

\begin{proof}
Suppose that we have a $k$-colouring $\lambda \colon \mathcal{F}_W \to \{1, \ldots, k\}$. We can define a new colouring $\bar{\lambda}$ such that:
\begin{itemize}
    \item on a facet $F \in \mathcal{F}_{W'_+}\setminus{M}$ it takes value $\lambda(F)$;
    \item on a facet $F \in \mathcal{F}_{W'_-}\setminus{M}$ it takes value $\lambda \circ s(F)$.
\end{itemize}
We now show that the map $\bar{\lambda}$ assigns to adjacent facets different colours. Let $F,G$ be two adjacent facets in $\mathcal{F}_W$. From Remark \ref{rem:construction}, they both belong to $\mathcal{F}_{W'_+}\setminus{M}$ or to $\mathcal{F}_{W'_-}\setminus{M}$. In the first case, $\bar{\lambda}$ takes the same value as $\lambda$, hence they have different image. In the second case, $F$ and $G$ are adjacent if and only if $s(F)$ and $s(G)$ are, hence \[\bar{\lambda}(F) = \lambda \circ s(F) \neq \lambda \circ s(G) = \bar{\lambda}(G).\]

The new colouring $\bar{\lambda}$ could be not surjective onto $\{1, \ldots, k\}$. In this case, we fix a bijection $b$ from the image of $\bar{\lambda}$ and $ \{1, \ldots, h\}$ and consider $b \circ \bar{\lambda}$.
\end{proof}

\begin{remark}
Note that the manifold $\mathcal{M}_\lambda$ obtained from a symmetric $k$-colouring of $W$ is isometric to the manifold $\mathcal{M}_\mu$ obtained by the colouring $\mu$ on $W'$ such that $\mu(F)=\lambda(F)$ if $F\neq M$ (where $F$ is seen both as a facet of $W'$ and as $W'_+ \cong W'$) and $\mu(M)=k+1$. Indeed, if we mirror $W'$ along $M$ (which was given a different colour than the rest of $W'$) and colour its copy with the same colours on corresponding facets, we get back the symmetric colouring $\lambda$ on $W$.
\end{remark}

Let $M$ be a dodecahedral manifold, $W$ the manifold with corners provided by Proposition \ref{prop:dodecahedral-embed} and $N$ the manifold obtained by a symmetric $k$-colouring of $W$. We now describe how the copies of $M$ are located in $N$. 

\begin{lemma}\label{lemma:dodecahedral-separation}
The disconnected manifold $M\times \mathbb{Z}_2 ^k$ embeds totally geodesically in $N$ as a separating submanifold. More precisely, the image of $M\times \mathbb{Z}_2 ^k$ separates $N$ in two isometric connected components $N_+$ and $N_-$. 
\end{lemma}

\begin{proof}
We have that $M\times \mathbb{Z}_2 ^k$ embeds totally geodesically in $W\times \mathbb{Z}_2 ^k$. We also have the quotient  $W\times \mathbb{Z}_2 ^k  \xrightarrow{\pi} (W\times \mathbb{Z}_2 ^k)/_\sim=N$, which identifies facets of $ W \times \mathbb{Z}_2 ^k$. Since $M\times \mathbb{Z}_2 ^k$ lies in $\mathrm{int}(W)\times \mathbb{Z}_2 ^k$, we have that it embeds geodesically in $N$ as well. From now on we will identify $M\times \mathbb{Z}_2 ^k$ with its image in $N$.

It follows by the construction of $W$ (see Remark \ref{rem:construction}) that $M$ separates $W$ in two isometric copies of $W'\setminus{M}$ and therefore $M\times \mathbb{Z}_2 ^k$ separates $W\times \mathbb{Z}_2 ^k$ in two isometric copies of $(W'\setminus{M})\times \mathbb{Z}_2 ^k$. Of course the image of these two copies via $\pi$ is exactly the complement of $M\times \mathbb{Z}_2 ^k$ inside $N$ and since they are saturated sets for the $\mathbb{Z}_2 ^k$ action their images are disjoint open sets in $W$. In order to conclude the proof we are then left to prove that their images are connected and isometric. This is a consequence of the colouring being symmetric. In fact each colouring of $W$ induces, by restriction, a colouring of each of the two copies of $W'\setminus{M}$. If the colouring is symmetric, then the induced colourings on the the two copies of $W'\setminus{M}$ coincide. By definition, the images of the two copies of $(W'\setminus{M})\times \mathbb{Z}_2 ^k$ in $W$ are exactly the manifolds obtained by colouring $W'\setminus{M}$ with these induced colourings, and therefore they are connected and isometric.
\end{proof}

\begin{remark}
The construction described in the proof above can also be visualized in Figure \ref{im:symm_col} and \ref{im:embedding} where the geodesic hypersurface is the purple circle. In this case, we have that $N_+ \cong N_-$ and $N$ is the mirror of $N_+$ along its boundary.
\end{remark}

\begin{figure}[H]

  \centering
  \includegraphics[scale=0.2]{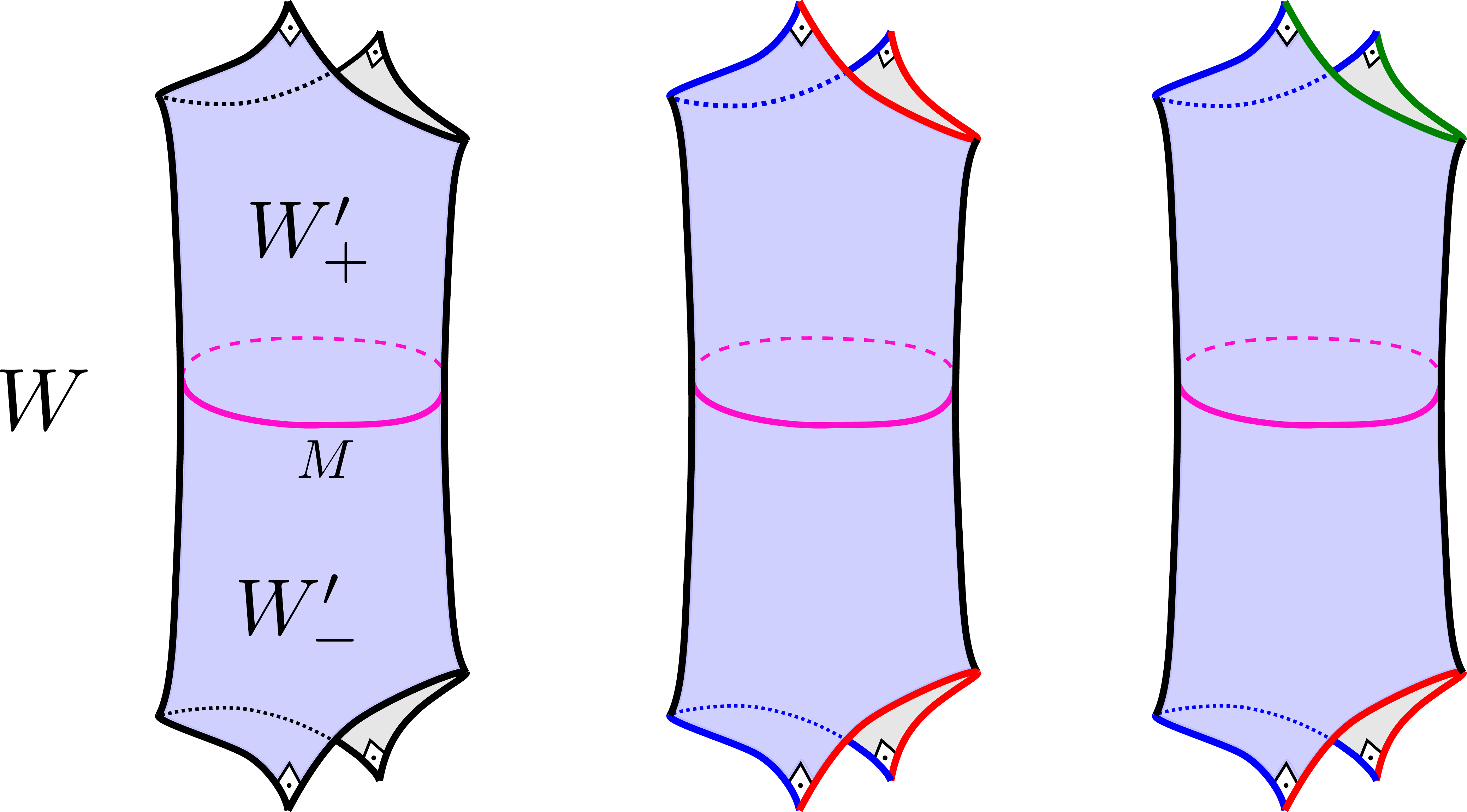}
\nota{The manifold $W$ built in Proposition \ref{prop:dodecahedral-embed}, with a symmetric and a non-symmetric colouring (see Definition \ref{defn:symm_col}). The manifold obtained using the symmetric colouring is isometric to the one in Figure  \ref{im:colourings}. See also Figure \ref{im:embedding}.}
\label{im:symm_col}
\end{figure}

\begin{figure}[H]

  \centering
  \includegraphics[scale=0.18]{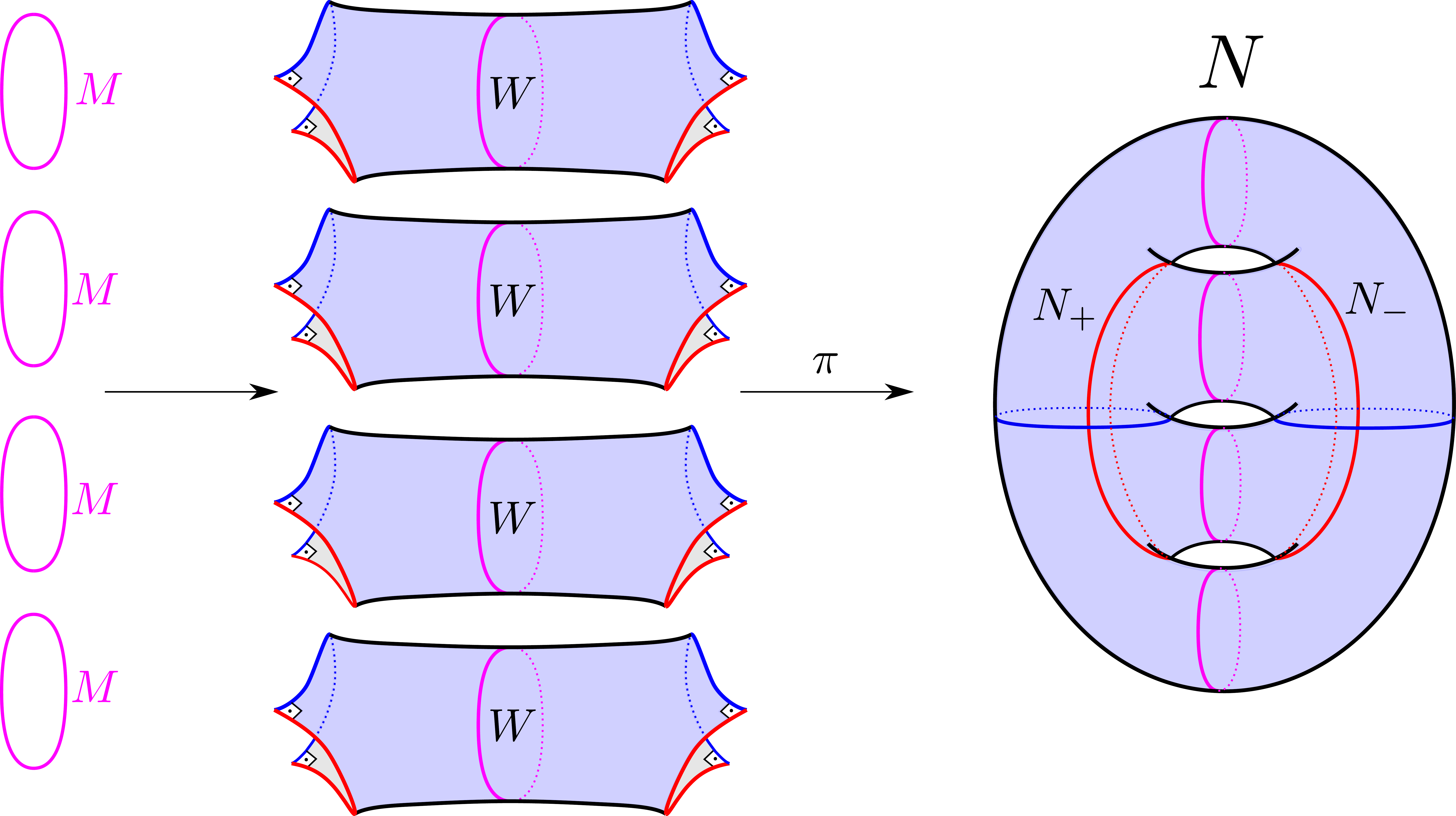}
\nota{A schematic picture of the embedding of the separating dodecahedral space, described in the proof of Lemma \ref{lemma:dodecahedral-separation}. }
\label{im:embedding}
\end{figure}

We are now ready to prove that if $M$ is an $L$-space the manifold $N$ contains a separating $L$-space:

\begin{prop}\label{Prop:separating connsum}
Suppose that the dodecahedral manifold $M$ is an $L$-space. Then the manifold $N$ can be written as $N=N_1\cup_{M'}N_2$, where $M'$ is an $L$-space and $b_2^+(N_i)\geq 1$ for $i=1,2$. In particular, by Proposition \ref{separating L-space}, the Seiberg-Witten invariants of $N$ all vanish.
\end{prop}
\begin{proof}
We use the same notations of Lemma \ref{lemma:dodecahedral-separation} and we divide the proof in few steps.
\begin{itemize}
    \item \emph{Step 1: $b_2^+(N)\geq 2$.}
    
    Recall that the Euler characteristic of a closed, connected, orientable $4$-manifold satisfies
    $$
    \chi=2-2b_1+b_2
    $$
    and that $b_2=b_2^++b_2^-$. Moreover, as a consequence of the Hirzebruch signature formula and \cite[Theorem 3]{C}, hyperbolic $4$-manifolds have signature zero and therefore $b_2^+=b_2^-$; hence, the Euler characteristic of $N$ satisfies
    $$
    \chi(N)=2(1-b_1(N)+b_2^+(N)).
    $$
    This formula implies that if $\chi(N)>4$ then $b_2^+(N)\geq 2$. In order to conclude we note that if $N$ is tessellated into $n$ 120-cells, then 
    \[ \chi(N)= \frac{17}{2}\cdot n. \]
    This can be proved by using the notions of \emph{orbifold covering} and \emph{characteristic simplex}, see \emph{e.g.}\ \cite[Section 1.4]{MartelliSurvey}. As a consequence, we have that $\chi(N) >8$.
    
    \item \emph{Step 2: $b_2^+(N_+)=b_2^+(N_-)\geq 1$ and $b_2^+(N_+)+b_2^+(N_-)=b_2^+(N)$.}
    
    Recall from the construction of $N$ and from the proof of Lemma \ref{lemma:dodecahedral-separation} that $N$ is obtained by gluing together $N_+$ and $N_-$ along their boundaries (that are made up by $2^k$ disjoint copies of the dodecahedral manifold $M$). Also recall that $N_+$ is isometric, and therefore diffeomorphic, to $N_-$. This implies that $b_2^+(N_+)=b_2^+(N_-)$. Since $N_+$ and $N_-$ are glued along a disjoint union of rational homology spheres, by applying the Mayer-Vietoris sequence we deduce that $H_2(N,\Q)\cong H_2(N_+,\Q)\oplus H_2(N_-,\Q)$ and that 
    $$
    b_2^+(N)=b_2^+(N_+)+b_2^+(N_-).
    $$
    Since $b_2^+(N)\geq 2$ we also deduce that $b_2^+(N_+)\geq 1$ and $b_2^-(N_-)\geq 1$.
    \item \emph{Step 3: there exists an $L$-space $M'$ such that $N=N_1\cup_{M'}N_2$.}
    
    The $L$-space $M'$ is diffeomorphic to the connected sum of copies of $M$ and $\overline{M}$, where $\overline{M}$ denotes $M$ with the opposite orientation. This operation of connected sum can be performed inside $N$ in the following way: we label the boundary components of $N_+$ with numbers $\{1,2\dots, 2^k\}$ and we consider $(2^k-1)$ pairwise disjoint properly embedded arcs $\alpha_1,\dots, \alpha_{2^k-1}$ in $N_+$ so that $\alpha_i$ connects the $i$-th and the $(i+1)$-th boundary components of $N_+$. If we denote with $U$ a tubular neighbourhood of these arcs in $N_+$ we have that $N=N_1\cup_\partial N_2$ where $N_1=N_-\cup U$ and $N_2=N_+\setminus{U}$, and where $\partial N_1=\partial N_2$ is exactly the connected sum $M'$. Since $M$ is an $L$-space, $M'$ is an $L$-space as a consequence of \cite[Proposition 6.1]{OS2}. See Figure \ref{Separating_conn_Lspace} for a schematic picture of this construction. By studying the orientations induced on the copies of $W$ embedded in $N$, one can state more precisely that \[M'\cong 2^{k-1}M \# 2^{k-1}\overline{M},\]  
    even if we do not need this for our construction.

\begin{figure}[H]

  \centering
  \includegraphics[scale=0.5]{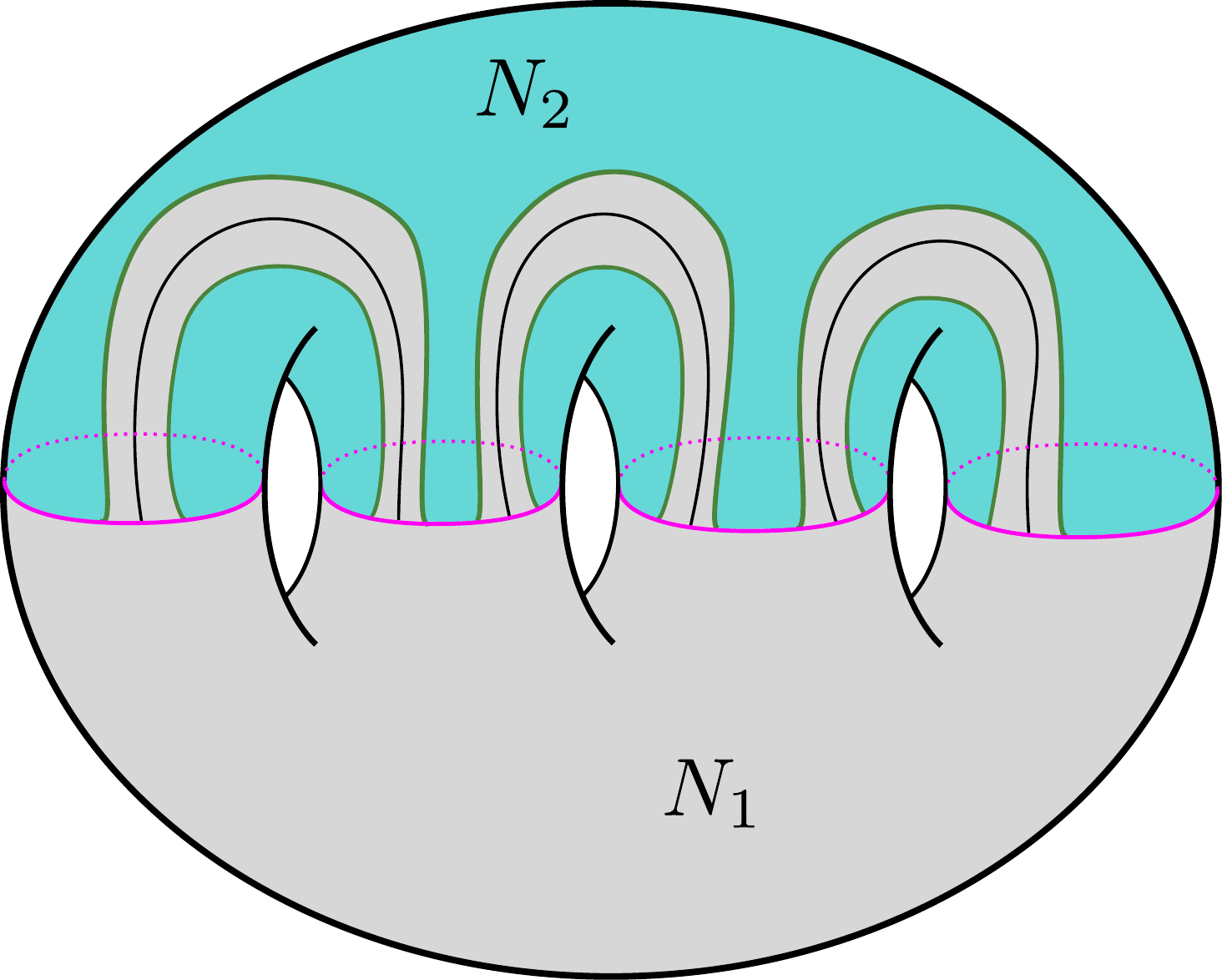}
\nota{A schematic picture of the construction presented in the proof of \emph{Step 3}. }
\label{Separating_conn_Lspace}
\end{figure}
    \item \emph{Step 4: $b_2^+(N_1)=b_2^+(N_2)\geq 1$.}
    
    By virtue of \emph{Step 2}, to prove that
    $$
    b_2^+(N_1)=b_2^+(N_2)\geq 1
    $$
    it will be sufficient to show that $b_2^+(N_1)=b_2^+(N_-)$ and $b_2^+(N_2)=b_2^+(N_+)$.
    The manifold $N_1$ is obtained by gluing $4$-dimensional $1$-handles to $N_-$, that is to say, by gluing copies of $\mathbb{D}^1\times \mathbb{D}^3$ to $N_-$ along $\partial \mathbb{D}^1\times \mathbb{D}^3$. Since the gluing regions have vanishing first and second homology groups, it is a consequence of the Mayer-Vietoris sequence that 
    $$
    H_2(N_1,\Z)=H_2(N_-,\Z)
    $$
    and that 
    $$
    b_2^+(N_1)=b_2^+(N_-).
    $$
    
    We are left to prove that $b_2^+(N_2)=b_2^+(N_+)$ holds. To do this we notice that since $N$ is obtained by gluing $N_+$ and $N_-$ along a rational homology sphere, we can apply the same reasoning of \emph{Step 2} to deduce that 
    $$
    b_2^+(N)=b_2^+(N_1)+b_2^+(N_2), 
    $$
    and as a consequence of $b_2^+(N_1)=b_2^+(N_-)$ we have that
    $$
    b_2^+(N_2)= b_2^+(N)-b_2^+(N_1)= b_2^+(N)-b_2^+(N_-)=b_2^+(N_+).
    $$
\end{itemize}
This concludes the proof.
\end{proof}

\subsection{Some concrete examples}\label{sec:concrete_examples}
To conclude we now use the theory and the construction introduced in the previous section to present some explicit examples of hyperbolic 4-manifolds fulfilling the requirements of Question \ref{qst1}. Recall from Theorem \ref{thm:classdodman} that we have six dodecahedral manifolds in the census $\mathscr{D}$ that are $L$-spaces. To produce our examples we want to apply Proposition \ref{Prop:separating connsum} and to do this we need to find colourings of the manifols with corners containing our dodecahedral $L$-spaces.
Before trying to find colourings, we have to check whether in the cases of our interests the construction of Proposition \ref{prop:dodecahedral-embed} yields manifolds with embedded facets (recall Example \ref{example_non_embedded}). Notice that by \cite{Goerner}, the six manifolds in $\mathscr{D}$ that are $L$-spaces admit only one tessellation in right-angled dodecahedra up to combinatorial isomorphism. We have the following:

\begin{prop}\label{prop:tackle}
Let $M$ be one of the six dodecahedral $L$-spaces in $\mathscr{D}$ (see Table \ref{table:dodecahedralmanifolds}). Then the manifold $W$ built as in Remark \ref{rem:construction} using its tessellation in dodecahedra has embedded facets if and only if $M$ has index $11$ or $28$.
\end{prop}
\begin{proof}
The proof is computer-based. We build the manifolds following the instructions of Remark \ref{rem:construction} starting with all the six dodecahedral $L$-spaces in $\mathscr{D}$ and we find non-embedded facets when performing this construction starting with the manifolds indexed with 0, 2, 8, 15. The code that we used can be found in \cite{codice}.
\end{proof}

Let $W_{11}$ (resp.\ $W_{28}$) be the manifold with corners built as in Remark \ref{rem:construction} starting with the dodecahedral manifold in $\mathscr{D}$ with index $11$ (resp.\ $28$). We have the following:
\begin{lemma}
The manifolds with corners $W_{11}$ and $W_{28}$ admit a symmetric 6-colouring.
\end{lemma}
\begin{proof}
The proof is computer-based. To search for a colouring of a manifold with corners $W$, we build a graph $G_W$ in the following way: we take one vertex for each facet and we add one edge between two vertices if the corresponding facets are adjacent.

Once we have $G_W$, finding a colouring for the manifold with corners $W$ is equivalent to finding a colouring of the graph $G_W$. The problem of colouring a graph is well known and Sage \cite{sagemath} provides a natural environment to search for minimal colouring of graphs. In our cases, the graph $G_{W_{11}}$ is made out by two equivalent connected components with 334 vertices, one corresponding to the facets of $W'_+$ and one corresponding to the ones of $W'_-$. In order to find a symmetric colouring of $W_{11}$ it is sufficient to colour only one of these two components and then extend the colouring as in the proof of Lemma \ref{lem:symm_colouring}. By using the  Mixed Integer Linear Programming solver CPLEX \cite{cplex} we find a 6-colouring in less than 5 minutes. We are also able to prove that $G_{W_{11}}$ is not 5-colourable.

The same holds also for $G_{W_{28}}$. We just point out that the fact that these graphs have the same number of vertices is not a casuality: this number in fact depends only on the number $n$ of dodecahedra in the tessellation of the dodecahedral 3-manifold. Namely, the total number of vertices in $G_W$ is \[ 2 \cdot n \cdot \tonde*{\frac{20}{8} + \frac{12}{2} + 30 + 12 + 20 + 12 + 1 }, \]where this formula descends from the discussion in \cite[Proof of Lemma 7]{M}.
\end{proof}

From now on we choose one specific 6-colouring for $W_{11}$ (resp.\ $W_{28}$), that we denote by $\lambda_{11}$ (resp.\ $\lambda_{28}$) and that can be found in \cite{codice}. We denote the hyperbolic 4-manifold obtained by this coloured manifold with corners (recall Proposition \ref{prop:corner-angle colouring}) with $\mathcal{N}_{11}$ (resp.\ $\mathcal{N}_{28}$). The following holds:

\begin{theorem}\label{thm:explicit_manifolds}
The manifold $\mathcal{N}_{11}$ (resp.\ $\mathcal{N}_{28}$):
\begin{enumerate}
    \item is a connected, orientable, closed, hyperbolic 4-manifold;
    \item is tessellated in $2^9$ right-angled hyperbolic 120-cells;
    \item can be written as $N_1\cup_{M'}N_2$, where $M'$ is an $L$-space and $b_2^+(N_i)\geq 1$ for $i=1,2$. In particular, its Seiberg-Witten invariants all vanish;
    \item has Betti numbers with coefficients in $\mathbb{R}$ and $\mathbb{Z}_2$ as described in Table \ref{table:homology_N_11} (resp.\ Table \ref{table:homology_N_28}).
\end{enumerate}
\end{theorem}

\begin{proof}
 Point 1 is a direct consequence of Proposition \ref{prop:corner-angle colouring}. The same proposition also tells us that $\mathcal{N}_{11}$ (resp.\ $\mathcal{N}_{28}$) is tessellated in $2^6$ copies of $W_{11}$ (resp.\ $W_{28}$); since the latter is tessellated in $8$ 120-cells, we also obtain Point 2. We point out that this tesselation is explicit; in particular, in \cite{codice} there is the list of maps that describe the gluings of the facets of these $2^9$ 120-cells. Point 3 is a direct consequence of Proposition \ref{Prop:separating connsum}. The proof of Point 4 is computer-based. Using its tessellation in 120-cells, we obtain a description of $\mathcal{N}_{11}$ (resp.\ $\mathcal{N}_{28}$) as a CW-complex. In this way we can also check that the Euler characteristic is consistent with the one obtained from the formula described in the proof of Proposition \ref{Prop:separating connsum}, Step 1. We can now compute the Betti numbers using cellular homology. In particular:
 \begin{itemize}
     \item with $\mathbb{Z}_2$ coefficients, we write all the matrices that represent the boundary maps of the cellular chain complex in the standard bases (the ones given by the $n$-cells in degree $n$). Computing their rank we recover the Betti numbers;
     \item with $\mathbb{R}$ coefficients, we write two matrices that represent the boundary maps of the cellular chain complex in the standard bases: the one from the 1-cells to the 0-cells and the one from the 2-cells to the 1-cells. Computing the rank of these matrices we are able to verify that $b_0=1$ and to determine the $b_1$, and we recover the other Betti numbers using the Poincaré Duality and the Euler characteristic. Using these matrices and the Universal Coefficient Theorem, one should be able to find all the integral homology. However, the computation is too heavy for our computer resources.
 \end{itemize}
 The proof is complete.
\end{proof}

\begin{table}[H]
\begin{center}
\begin{tabular}{c || c | c | c | c | c}
$\mathcal{N}_{11}$ & $b_0$ & $b_1$ & $b_2$ & $b_3$ & $b_4$ \\
\hline
$\mathbb{R}$ & 1 & 725 & 5800 & 725 & 1 \\
$\mathbb{Z}_2$ & 1 & 746 & 5842 & 746 & 1\\

\end{tabular}
\vspace{.2 cm}
\nota{The Betti numbers of $\mathcal{N}_{11}$.} 
\label{table:homology_N_11}
\end{center}
\end{table}

\begin{table}[H]
\begin{center}
\begin{tabular}{c || c | c | c | c | c}
$\mathcal{N}_{28}$ & $b_0$ & $b_1$ & $b_2$ & $b_3$ & $b_4$ \\
\hline
$\mathbb{R}$ & 1 & 741 & 5832 & 741 & 1 \\
$\mathbb{Z}_2$ & 1 & 769 & 5888 & 769 & 1\\

\end{tabular}
\vspace{.2 cm}
\nota{The Betti numbers of $\mathcal{N}_{28}$.} 
\label{table:homology_N_28}
\end{center}
\end{table}

\subsubsection{Generalised colourings}\label{sec:generalised_colouring}

In this section we briefly describe a well-known generalisation of the notion of \emph{colouring} (see \cite{FKS} for a complete discussion). Let $W$ be a compact $n$-manifold with corners. Let $S$ be $\mathbb{Z}_2^m$, and let $e_1,\ldots,e_m$ be its standard basis. We say that an element $v \in S$ is \emph{odd} if it is a sum of an odd number of elements in the standard basis, and \emph{even} otherwise.

\begin{definition}
A \emph{generalised $m$-colouring} of $W$ is a map $\rho \colon \mathcal{F}_W \to S$ such that:
\begin{itemize}
    \item the elements in $\{ \rho(F) \}_{F \in \mathcal{F}_W}$ generate $S$;
    \item if $F_{i_1}, \ldots , F_{i_r}$ share a common subface, their images through $\rho$ are linearly independent.
\end{itemize}
\end{definition}

Given a generalised $m$-colouring $\rho$, we can define the topological space $\mathcal{M}_\rho=(W\times S)/_\sim$, where distinct $W\times \{u\}$, $W\times \{v\}$ are glued along the identity on a facet $F\in \mathcal{F}_W$ if $u-v=\rho(F)$. The analogous of Proposition \ref{prop:corner-angle colouring} holds, and can be shown combining \cite[Proposition 6]{M} and \cite[Lemma 2.4]{KMT}:
\begin{prop}
The resulting $\mathcal{M}_\rho$ is a (possibly non-orientable) connected hyperbolic $n$-manifold tessellated into $2^m$ copies of $W$. If the elements in $\{ \rho(F) \}_{F \in \mathcal{F}_W}$ are all odd, it is orientable.
\end{prop}

\begin{remark}\label{rem:col_to_gen_col}
Given a $k$-colouring $\lambda \colon W \to \{ 1,\ldots,k  \}$, one can associate to it a generalised $k$-colouring $\rho_{\lambda}$ by taking $S=\mathbb{Z}_2^k$ and defining $\rho(F)=e_{\lambda(F)}$. This is a generalised colouring because in a compact manifold with corners, facets that share a common subface are pairwise adjacent (see Example \ref{ex:corner-angle colouring}).
The manifolds $\mathcal{M}_\lambda$ and $\mathcal{M}_{\rho}$ are naturally isometric.
\end{remark}

Manifolds obtained by colourings are easier to visualize than those obtained by generalised colourings. On the other hand, generalised colourings often give the chance to obtain manifolds tessellated with a lower number of copies of $W$, as the following proposition shows (see also \cite[Proposition 3.1.16]{LeonardoPHDThesis}):

\begin{prop}
Let $\lambda$ be a $k$-colouring of a compact $n$-manifold with corners $W$. Let $\rho_\lambda \colon \mathcal{F}_W \to \mathbb{Z}_2^{k-1}$ be the map:
 \begin{equation}
    \rho(F) =
    \begin{cases*}
      e_{\lambda(F)} & if $\lambda(F) \neq k$ \\
      e_1 + \ldots + e_{k-1}        & if $\lambda(F) = k$
    \end{cases*}.
  \end{equation}
  If $k$ is even and $k > n$, $\rho_\lambda$ is a generalised $(k-1)$-colouring such that all the elements in $\{ \rho_\lambda(F) \}_{F \in \mathcal{F}_W}$ are odd.
\end{prop}
\begin{proof}
The set $\{ \rho_\lambda(F) \}_{F \in \mathcal{F}_W}$ generates $\mathbb{Z}_2^{k-1}$ because it contains $e_1,\ldots,e_{k-1}$.
To finish the proof we just need to show that if $F_{i_1}, \ldots , F_{i_r}$ share a common subface, their images through $\rho_\lambda$ are linearly independent.
Since the dimension of $W$ is $n$, we know that $r \leq n < k$. Since $\lambda$ is a colouring, we know that the set $A= \{\lambda(F_{i_1}), \ldots , \lambda(F_{i_r}) \}$ contains $r$ distinct elements. Then we can conclude that:

\begin{itemize}
    \item if $k \not \in A$, $\rho_\lambda(F_{i_1}), \ldots , \rho_\lambda(F_{i_r})$ are independent because they are part of a basis, see also Remark \ref{rem:col_to_gen_col};
    \item otherwise, there is an element of  $\{ 1,\ldots,k  \}$ that is not in $A$. It is then easy to show that $\rho_\lambda(F_{i_1}), \ldots , \rho_\lambda(F_{i_r})$ are independent.
\end{itemize}
The proof is complete.
\end{proof}

Using the previous proposition, it is easy to obtain a generalised 5-colouring of $W_{11}$ (resp.\ $W_{28}$) starting with $\lambda_{11}$ (resp.\ $\lambda_{28}$) that produces an orientable manifold $\mathcal{M}_{11}$ (resp.\ $\mathcal{M}_{28}$) that is double-covered by $\mathcal{N}_{11}$ (resp.\ $\mathcal{N}_{28}$) and satisfies:

\begin{prop}
The manifold $\mathcal{M}_{11}$ (resp.\ $\mathcal{M}_{28}$):
\begin{enumerate}
    \item is a connected, orientable, closed, hyperbolic 4-manifold;
    \item is tessellated in $2^8$ right-angled hyperbolic 120-cells;
    \item can be written as $N_1\cup_{M'}N_2$, where $M'$ is an $L$-space and $b_2^+(N_i)\geq 1$ for $i=1,2$. In particular, its Seiberg-Witten invariants all vanish;
    \item has Betti numbers with coefficients in $\mathbb{R}$ and $\mathbb{Z}_2$ as described in Table \ref{table:homology_M_11} (resp.\ Table \ref{table:homology_M_28}).
\end{enumerate}
\end{prop}

\begin{proof}
Everything works exactly as in proof of Theorem \ref{thm:explicit_manifolds}.
\end{proof}

\begin{table}[H]
\begin{center}
\begin{tabular}{c || c | c | c | c | c}
$\mathcal{M}_{11}$ & $b_0$ & $b_1$ & $b_2$ & $b_3$ & $b_4$ \\
\hline
$\mathbb{R}$ & 1 & 37 & 2248 & 37 & 1 \\
$\mathbb{Z}_2$ & 1 & 707 & 3588 & 707 & 1\\

\end{tabular}
\vspace{.2 cm}
\nota{The Betti numbers of $\mathcal{M}_{11}$.} 
\label{table:homology_M_11}
\end{center}
\end{table}

\begin{table}[H]
\begin{center}
\begin{tabular}{c || c | c | c | c | c}
$\mathcal{M}_{28}$ & $b_0$ & $b_1$ & $b_2$ & $b_3$ & $b_4$ \\
\hline
$\mathbb{R}$ & 1 & 53 & 2280 & 53 & 1 \\
$\mathbb{Z}_2$ & 1 & 713 & 3600 & 713 & 1\\

\end{tabular}
\vspace{.2 cm}
\nota{The Betti numbers of $\mathcal{M}_{28}$.} 
\label{table:homology_M_28}
\end{center}
\end{table}

\section{Questions and further developments}\label{sec:questions_fur_dev}

\textbf{Applicability of the algorithm.}
 Despite the potential problems that the algorithm described in Section \ref{section:proving_L_space} may encounter, in practice it was pretty fast in proving that the six manifolds of Theorem \ref{thm:classdodman} are $L$-spaces. In some cases, it helped avoiding some curves at the beginning of the algorithm to converge faster (see \cite{codice} for more details).
 
 It would be interesting to use it to study larger families of manifolds. The greatest bottleneck appears to be the computation of the Turaev torsion, whose complexity seems to grow very fast with the number of tetrahedra in the ideal triangulation of the \qht{}.
 
 The authors were made aware by Nathan Dunfield in a mail exchange that by using a similar approach it was possible to show that the Seifert-Weber manifold is an $L$-space with bare hands. This is a hyperbolic manifold tessellated by one hyperbolic dodecahedron with $\frac{2}{5} \pi$ dihedral angles, and was proved to be an $L$-space in the context of the monopole Floer homology in \cite{LL} with completely different methods. Also our algorithm confirms that this is the case.
    
\textbf{Limitations of the algorithm.}
The first step of the algorithm consists in drilling a curve from the rational homology sphere that one wants to study hoping to find a rational homology solid torus that is Floer simple, but we do not know if such a curve exists, even when if we start with an $L$-space. So we leave here the following question:
    \begin{question}\label{question:floer_simple}
    Let $M$ be an $L$-space. Does there exist a Floer simple rational homology solid torus $Y$ such that $M$ is a Dehn filling on $Y$?
    \end{question} 
A more specific question that applies to our algorithm would be:
\begin{question}\label{question:floer_simple_geodesic}
Let $M$ be a hyperbolic $L$-space. Does there exist a geodesic $\gamma$ such that drilling $\gamma$ out of $M$ yields a Floer simple rational homology solid torus?
\end{question}

\bibliographystyle{alpha}
\bibliography{Biblio}

\newcommand{\etalchar}[1]{$^{#1}$}
\begin{thebibliography}{KMOS07}

\bibitem[AL20]{AL}
Ian Agol and Francesco Lin.
\newblock Hyperbolic four-manifolds with vanishing {Seiberg-Witten} invariants.
\newblock {\em Characters in Low-Dimensional Topology}, 760:1--8, 2020.

\bibitem[BBP{\etalchar{+}}21]{regina}
Benjamin~A. Burton, Ryan Budney, William Pettersson, et~al.
\newblock Regina: Software for low-dimensional topology.
\newblock {\tt http://\allowbreak regina-normal.\allowbreak github.\allowbreak
  io/}, 1999--2021.

\bibitem[BGW13]{BGW}
Steven Boyer, Cameron~McA Gordon, and Liam Watson.
\newblock On {L}-spaces and left-orderable fundamental groups.
\newblock {\em Mathematische Annalen}, 356(4):1213--1245, 2013.

\bibitem[Bow16]{Bow}
Jonathan Bowden.
\newblock Approximating {$C^0$}-foliations by contact structures.
\newblock {\em Geometric and Functional Analysis}, 26(5):1255--1296, 2016.

\bibitem[BP92]{BenedettiPetronio}
Riccardo Benedetti and Carlo Petronio.
\newblock {\em Lectures on hyperbolic geometry}.
\newblock Springer Science \& Business Media, 1992.

\bibitem[Bur14]{B}
Benjamin~A Burton.
\newblock The cusped hyperbolic census is complete.
\newblock {\em arXiv preprint arXiv:1405.2695}, 2014.

\bibitem[CDGW]{snappy}
Marc Culler, Nathan~M. Dunfield, Matthias Goerner, and Jeffrey~R. Weeks.
\newblock Snap{P}y, a computer program for studying the geometry and topology
  of $3$-manifolds.
\newblock Available at \url{http://snappy.computop.org} (28/06/2022).

\bibitem[CGH11]{CGH}
Vincent Colin, Paolo Ghiggini, and Ko~Honda.
\newblock Equivalence {of Heegaard Floer} homology and embedded contact
  homology via open book decompositions.
\newblock {\em Proceedings of the National Academy of Sciences},
  108(20):8100--8105, 2011.

\bibitem[Che55]{C}
Shiing-Shen Chern.
\newblock On curvature and characteristic classes of a {R}iemann manifold.
\newblock In {\em Abhandlungen aus dem Mathematischen Seminar der
  Universit{\"a}t Hamburg}, volume~20, pages 117--126. Springer, 1955.

\bibitem[Cod22]{codice}
Code.
\newblock {Dodecahedral L-spaces and hyperbolic 4-manifolds}.
\newblock \url{https://doi.org/10.7910/DVN/A9WKAG}, 2022.

\bibitem[Cpl21]{cplex}
IBM~ILOG Cplex.
\newblock V20.1: User’s manual for cplex.
\newblock {\em International Business Machines Corporation}, 2021.

\bibitem[Dun20]{D}
Nathan~M. Dunfield.
\newblock Floer homology, group orderability, and taut foliations of hyperbolic
  3--manifolds.
\newblock {\em Geometry \& Topology}, 24:2075--2125, 2020.

\bibitem[Fer21]{LeonardoPHDThesis}
Leonardo Ferrari.
\newblock {\em Hyperbolic Manifolds and Coloured Polytopes}.
\newblock PhD thesis, Università di Pisa, 2021.

\bibitem[FKS21]{FKS}
Leonardo Ferrari, Alexander Kolpakov, and Leone Slavich.
\newblock Cusps of hyperbolic 4-manifolds and rational homology spheres.
\newblock {\em Proceedings of the London Mathematical Society},
  123(6):636--648, 2021.

\bibitem[Goe16]{Goerner}
Matthias Goerner.
\newblock A census of hyperbolic {P}latonic manifolds and augmented knotted
  trivalent graphs.
\newblock {\em arXiv preprint arXiv:1602.02208}, 2016.

\bibitem[Hom17]{H}
Jennifer Hom.
\newblock A survey on {Heegaard Floer} homology and concordance.
\newblock {\em Journal of Knot Theory and Its Ramifications}, 26(02):1740015,
  2017.

\bibitem[Juh15]{J}
Andr{\'a}s Juh{\'a}sz.
\newblock A survey of {Heegaard Floer} homology.
\newblock In {\em New ideas in low dimensional topology}, pages 237--296. World
  Scientific, 2015.

\bibitem[KLT20a]{KLT1}
{\c{C}}a{\u{g}}atay Kutluhan, Yi-Jen Lee, and Clifford Taubes.
\newblock {HF= HM, I}: {Heegaard Floer homology and Seiberg--Witten Floer
  homology}.
\newblock {\em Geometry \& Topology}, 24(6):2829--2854, 2020.

\bibitem[KLT20b]{KLT2}
{\c{C}}a{\u{g}}atay Kutluhan, Yi-Jen Lee, and Clifford Taubes.
\newblock {HF= HM, II}: {Reeb orbits and holomorphic curves for the
  ech/Heegaard Floer correspondence}.
\newblock {\em Geometry \& Topology}, 24(6):2855--3012, 2020.

\bibitem[KLT20c]{KLT3}
{\c{C}}a{\u{g}}atay Kutluhan, Yi-Jen Lee, and Clifford Taubes.
\newblock {HF= HM, III}: {holomorphic curves and the differential for the
  ech/Heegaard Floer correspondence}.
\newblock {\em Geometry \& Topology}, 24(6):3013--3218, 2020.

\bibitem[KLT21a]{KLT4}
{\c{C}}a{\u{g}}atay Kutluhan, Yi-Jen Lee, and Clifford Taubes.
\newblock {HF= HM, IV}: {The Seiberg--Witten Floer homology and ech
  correspondence}.
\newblock {\em Geometry \& Topology}, 24(7):3219--3469, 2021.

\bibitem[KLT21b]{KLT5}
{\c{C}}a{\u{g}}atay Kutluhan, Yi-Jen Lee, and Clifford Taubes.
\newblock {HF= HM, V}: {Seiberg--Witten Floer} homology and handle additions.
\newblock {\em Geometry \& Topology}, 24(7):3471--3748, 2021.

\bibitem[KM07]{KM}
Peter~B. Kronheimer and Tomasz Mrowka.
\newblock {\em Monopoles and three-manifolds}, volume~10.
\newblock Cambridge University Press Cambridge, 2007.

\bibitem[KMOS07]{KMOS}
Peter Kronheimer, Tomasz Mrowka, Peter Ozsv{\'a}th, and Zolt{\'a}n Szab{\'o}.
\newblock Monopoles and lens space surgeries.
\newblock {\em Annals of mathematics}, pages 457--546, 2007.

\bibitem[KMT15]{KMT}
Alexander Kolpakov, Bruno Martelli, and Steven Tschantz.
\newblock Some hyperbolic three-manifolds that bound geometrically.
\newblock {\em Proc. Amer. Math. Soc.}, 143:9:4103--4111, 2015.

\bibitem[KR17]{KR}
William Kazez and Rachel Roberts.
\newblock {$C^0$} approximations of foliations.
\newblock {\em Geometry \& Topology}, 21(6):3601--3657, 2017.

\bibitem[KRS18]{KRS}
Alexander Kolpakov, Alan~W. Reid, and Leone Slavich.
\newblock Embedding arithmetic hyperbolic manifolds.
\newblock {\em Mathematical Research Letters}, 25:1305--1328, 2018.

\bibitem[LeB02]{L}
Claude LeBrun.
\newblock Hyperbolic manifolds, harmonic forms, and {Seiberg--Witten}
  invariants.
\newblock {\em Geometriae Dedicata}, 91(1):137--154, 2002.

\bibitem[LL22]{LL}
Francesco Lin and Michael Lipnowski.
\newblock Monopole {Floer Homology}, {Eigenform Multiplicities}, and the
  {Seifert--Weber Dodecahedral Space}.
\newblock {\em International Mathematics Research Notices}, 2022(9):6540--6560,
  2022.

\bibitem[Mar16a]{M}
Bruno Martelli.
\newblock Hyperbolic 3-manifolds that embed geodesically.
\newblock {\em https://arxiv.org/abs/1510.06325}, 2016.

\bibitem[Mar16b]{Martellibook}
Bruno Martelli.
\newblock An introduction to geometric topology.
\newblock {\em arXiv preprint arXiv:1610.02592}, 2016.

\bibitem[Mar18]{MartelliSurvey}
Bruno Martelli.
\newblock Hyperbolic four-manifolds.
\newblock In {\em Handbook of group actions, {III}}, pages 37--58.
  International Press of Boston, Inc., 2018.

\bibitem[OS04a]{OS}
Peter Ozsv{\'a}th and Zolt{\'a}n Szab{\'o}.
\newblock Holomorphic disks and genus bounds.
\newblock {\em Geometry \& Topology}, 8(1):311--334, 2004.

\bibitem[OS04b]{OS2}
Peter Ozsv{\'a}th and Zolt{\'a}n Szab{\'o}.
\newblock Holomorphic disks and three-manifold invariants: properties and
  applications.
\newblock {\em Annals of Mathematics}, pages 1159--1245, 2004.

\bibitem[OS04c]{OS3}
Peter Ozsv{\'a}th and Zolt{\'a}n Szab{\'o}.
\newblock Holomorphic disks and topological invariants for closed
  three-manifolds.
\newblock {\em Annals of Mathematics}, pages 1027--1158, 2004.

\bibitem[OS05]{OS1}
Peter Ozsv{\'a}th and Zolt{\'a}n Szab{\'o}.
\newblock On knot floer homology and lens space surgeries.
\newblock {\em Topology}, 44(6):1281--1300, 2005.

\bibitem[RR17]{RR}
Jacob Rasmussen and Sarah~Dean Rasmussen.
\newblock Floer simple manifolds and {$L$}-space intervals.
\newblock {\em Advances in Mathematics}, 322:738--805, 2017.

\bibitem[SW94a]{SW}
Nathan Seiberg and Edward Witten.
\newblock Electric-magnetic duality, monopole condensation, and confinement in
  n= 2 supersymmetric yang-mills theory.
\newblock {\em Nuclear Physics B}, 426(1):19--52, 1994.

\bibitem[SW94b]{SW1}
Nathan Seiberg and Edward Witten.
\newblock Monopoles, duality and chiral symmetry breaking in n= 2
  supersymmetric {QCD}.
\newblock {\em Nuclear Physics B}, 431(3):484--550, 1994.

\bibitem[SW10]{SarWan}
Sucharit Sarkar and Jiajun Wang.
\newblock An algorithm for computing some {Heegaard Floer} homologies.
\newblock {\em Annals of mathematics}, pages 1213--1236, 2010.

\bibitem[Tau94]{Tau}
Clifford~Henry Taubes.
\newblock The {Seiberg-Witten} invariants and symplectic forms.
\newblock {\em Mathematical Research Letters}, 1(6):809--822, 1994.

\bibitem[Tau10a]{Tau1}
Clifford~Henry Taubes.
\newblock Embedded contact {homology and Seiberg--Witten Floer cohomology I}.
\newblock {\em Geometry \& Topology}, 14(5):2497--2581, 2010.

\bibitem[Tau10b]{Tau2}
Clifford~Henry Taubes.
\newblock Embedded contact {homology and Seiberg--Witten Floer cohomology II}.
\newblock {\em Geometry \& Topology}, 14(5):2583--2720, 2010.

\bibitem[Tau10c]{Tau3}
Clifford~Henry Taubes.
\newblock Embedded contact {homology and Seiberg--Witten Floer cohomology III}.
\newblock {\em Geometry \& Topology}, 14(5):2721--2817, 2010.

\bibitem[Tau10d]{Tau4}
Clifford~Henry Taubes.
\newblock Embedded contact {homology and Seiberg--Witten Floer cohomology IV}.
\newblock {\em Geometry \& Topology}, 14(5):2819--2960, 2010.

\bibitem[Tau10e]{Tau5}
Clifford~Henry Taubes.
\newblock Embedded contact {homology and Seiberg--Witten Floer cohomology V}.
\newblock {\em Geometry \& Topology}, 14(5):2961--3000, 2010.

\bibitem[{The}22]{sagemath}
{The Sage Developers}.
\newblock {\em {S}ageMath, the {S}age {M}athematics {S}oftware {S}ystem}, 2022.
\newblock {\tt https://www.sagemath.org}.

\bibitem[Tur02]{T}
Vladimir Turaev.
\newblock {\em Torsions of 3-dimensional manifolds}, volume 208.
\newblock Birkh{\"a}user, 2002.

\bibitem[Ves87]{V1}
Andrei~Yurievich Vesnin.
\newblock Three-dimensional hyperbolic manifolds of {Löbell} type.
\newblock {\em Sibirskii Matematicheskii Zhurnal}, 28(5):50--53, 1987.

\bibitem[Ves10]{V2}
Andrei Vesnin.
\newblock Volumes and normalized volumes of right-angled hyperbolic polyhedra.
\newblock {\em Atti del Seminario Matematico e Fisico dell’ Università di
  Modena e Reggio Emilia}, 01 2010.

\bibitem[Wit94]{W}
Edward Witten.
\newblock {Monopoles and four manifolds}.
\newblock {\em Math. Res. Lett.}, 1:769--796, 1994.

\end{thebibliography}

\noindent\textsc{Dipartimento di Matematica \\
Università di Bologna \\
Piazza di Porta S. Donato, 5 \\
40126 Bologna BO, Italy \\}
\textit{Email address: \href{mailto:ludox73@gmail.com}{ludox73@gmail.com}}
\\ \\
\textsc{Institut de mathematiqués \\
Universite de Neuchàtêl \\
Rue Emile–Argand, 11 \\
2000 Neuchatel, Switzerland \\}
\textit{Email address: \href{mailto:leonardocpferrari@gmail.com}{leonardocpferrari@gmail.com} \\}
\\ \\
\textsc{Dipartimento di Matematica \\
Scuola Normale Superiore \\
P.za dei Cavalieri, 7 \\
56126 Pisa PI, Italy \\}
\textit{Email address: \href{mailto:diego.santoro95@gmail.com}{diego.santoro95@gmail.com}}

\clearpage

\SetKwComment{Comment}{/* }{ */}
\SetAlgoNoLine
\SetAlgoNoEnd

\appendix
\section{The algorithm} \label{appendix:alg}
In this appendix we explain in detail how the algorithm described in Section \ref{section:proving_L_space} works. The code can be found in \cite{codice}. The pseudocode presented here is written to improve the readability, not the speed of the algorithm; the implemented version is slightly different. Whenever we have a rational homology solid torus $Y$, we suppose it is given with a chosen basis for $H_1(\partial Y, \mathbb{Z})$; for this reason we always identify $Sl(Y)$ with $\mathbb{Q} \cup \{\infty\}$. We refer to Section \ref{section:L_space_intervals} for the notation.

The algorithm essentially goes back and forth alternating between these two functions:

\begin{itemize}
    \item \texttt{is\_certified\_L\_space($M$)}: takes as input $M$, a hyperbolic \qhs{}, and searches for a nice drilling: it returns $Y$ and $\mathcal{I}$, where $Y$ is a Turaev simple hyperbolic \qht{} such that the $\sfrac{1}{0}$ filling on $Y$ gives back $M$ and $\mathcal{I}$ is an interval in $Sl(Y)$ that contains $\sfrac{1}{0}$ and whose endpoints are elements in $\iota^{-1}(D_{>0}^\tau(Y))$ (in the case $D_{>0}^\tau(Y)=\emptyset$ we take $\mathcal{I}$ as $Sl(Y)\setminus{[l]}$, where $[l]$ is the homological longitude of $Y$);
    \item \texttt{small\_fillings}$(Y, \mathcal{I})$: takes as input $Y$, a Turaev simple hyperbolic \qht, and $\mathcal{I}$, an interval in $Sl(Y)$, and returns the two \qhs{} with smallest volume among the fillings $Y(\alpha)$ with $\alpha \in \mathcal{I}$. 
\end{itemize}

While these two functions operate, they try to identify all the manifolds they work with, with the hope to obtain information about their $L$-space value using the Dunfield census $\mathscr{Y}$. If we discover that two fillings with coefficient in $\mathcal{I}$ are $L$-spaces, we use Theorem \ref{theorem RR} to conclude that all the fillings with coefficients in $\mathcal{I}$ are $L$-spaces. Eventually, we hope we will conclude that the initial manifold is an $L$-space. 

We need some globally-defined variables:
\begin{itemize}
    \item \texttt{old\_mnfds}: a list of the manifolds that we already found. We want to avoid them, because otherwise the algorithm would enter an infinite loop. Its starting value is the empty list;
    \item \texttt{M\_C}: Max Coefficient, a positive integer. When we search for minimal volume fillings on $Y$, we search among the fillings $\sfrac{h}{k}$ with $|h|,|k| \leq$ \texttt{M\_C}.
\end{itemize}

\begin{algorithm}

\caption{\texttt{is\_certified\_L\_space}}\label{alg:one}
\KwData{$M$ a hyperbolic \qhs{}.}
\KwResult{\texttt{True} if a proof that $M$ is an $L$-space is found, \texttt{False} otherwise.}
--------------------------------------------------------------------------

\textbf{add} $M$ to \texttt{old\_mnfds}\;

$C = $ finite collection of simple closed curves in $M$, provided by SnapPy\;
$X= \{M \smallsetminus \text{the interior of a tubular neighborhood of } c\}_{c \in C}$\; \Comment{For every $Y$ in $X$, we fix a peripheral basis such that the filling $\sfrac{1}{0}$ on $Y$ gives back $M$.}

\For{$Y \in X$}{
\If{ $Y$ belongs to Dunfield census }{
\Return "$L$-space value of $M$ found using the census" }}
\texttt{found} $ = $ \texttt{False} \;
\While{\texttt{found} $ == $ \texttt{False}}{
\If{ there is no hyperbolic $Y$ in $X \smallsetminus$\texttt{old\_mnfds} }{
\Return \texttt{False}
}
$Y=$ smallest volume hyperbolic $Y \in X \smallsetminus$\texttt{old\_mnfds}\;
\eIf{$Y$ is Turaev simple}{
\texttt{found} $ = $ \texttt{True} \;
}
{
$X = X \smallsetminus \{ Y \}$\;
}

}
\texttt{Possible\_intervals} $=$ intervals in $Sl(Y)$ that contains $\sfrac{1}{0}$ and whose endpoints are elements in $\iota^{-1}(D_{>0}^\tau(Y))$ (in the case $D_{>0}^\tau(Y)=\emptyset$ we consider $Sl(Y)\setminus{[l]}$, where $[l]$ is the homological longitude of $Y$)\;
\uIf{\texttt{Possible\_intervals} $== \{ \mathcal{I}_1; \mathcal{I}_2\} $}{
\Return (\texttt{small\_fillings}$(Y, \mathcal{I}_1)$ \textbf{or} \texttt{small\_fillings}$(Y, \mathcal{I}_2)$) \;
}
\ElseIf{\texttt{Possible\_intervals} $== \{ \mathcal{I} \} $}{
\Return \texttt{small\_fillings}$(Y, \mathcal{I})$\;
}
\bigskip






\end{algorithm}

\begin{algorithm}

\caption{\texttt{small\_fillings}}\label{alg:two}
\KwData{$Y$ Turaev simple hyperbolic \qht{}; $\mathcal{I}$, an interval in $Sl(Y)$ that contains $\sfrac{1}{0}$ and whose endpoints are in $\iota^{-1}(D_{>0}^\tau(Y))$.}
\KwResult{\texttt{True} if a proof that $Y(\sfrac{1}{0})$ is an $L$-space is found, \texttt{False} otherwise.}
--------------------------------------------------------------------------

\textbf{add} $Y$ to \texttt{old\_mnfds}\;
\texttt{pairs} $= \{ (h,k) | h \in [-\texttt{M\_C}, \texttt{M\_C}], k \in [0, \texttt{M\_C}],  \gcd(h,k)=1  \}  $\;

$X= \{Y(\sfrac{h}{k}) |  (h,k) \in \texttt{pairs} , (h,k)\neq(-1,0), (h,k)\in \mathcal{I}\}$\;

\Comment{$X$ is a list more than a set: even if $Y(\sfrac{h}{k})$ is diffeomorphic to $Y(\sfrac{h'}{k'})$, we count them as two elements in $X$. For this reason we avoid the case $(h,k)=(-1,0)$. }

\texttt{L\_space\_found}=0\;

\For{$M \in X$}{
\If{$M$ is in Dunfield\_QHS\_Census }{
\eIf{"$L$-space value of $M$ found using the census" == \texttt{False}}{\Return \texttt{False}}
{ \texttt{L\_space\_found}=\texttt{L\_space\_found}+1\; $X = X \smallsetminus \{ M \}$\; }
}
}

\If{\texttt{L\_space\_found} $\geq 2$}{\Return \texttt{True}}
\If{\texttt{L\_space\_found} $== 1$}{
\If{ there is no hyperbolic $M$ in $X \smallsetminus$\texttt{old\_mnfds} }{
\Return \texttt{False}
}
$M= $ smallest volume hyperbolic $M \in X \smallsetminus$\texttt{old\_mnfds}\;
\Return \texttt{is\_certified\_L\_space}(M)
}

\If{\texttt{L\_space\_found} $== 0$}{
\If{ there are no two hyperbolic manifolds in $X \smallsetminus$\texttt{old\_mnfds} }{
\Return \texttt{False}
}
$M_1, M_2= $ two smallest volume hyperbolic manifolds in $X \smallsetminus$\texttt{old\_mnfds}\;
\Return (\texttt{is\_certified\_L\_space}($M_1$) \textbf{and} \texttt{is\_certified\_L\_space}($M_2$))\;
}
\bigskip

\end{algorithm}
\clearpage

\section{Reading the output of the algorithm}\label{appendix:example}

Here we comment how to read the output of the function \texttt{is\_certified\_L\_space}. Suppose we give the following commands:

\begin{lstlisting}[basicstyle=\small, language=Python]
M=CubicalOrientableClosedCensus(betti=0)[15]
val=is_certified_L_space(M, save_QHT=True, path_save_QHT="./proofs/RA_dod_15/")
print(val)
\end{lstlisting}

At line 2, we set the options \[ \texttt{save\_QHT=True, path\_save\_QHT="./proofs/RA\_dod\_15/"}; \] they mean that we want to save the $\mathbb{Q}HT$s that we find while the algorithm is running as SnapPy triangulations. These will be saved in the directory \texttt{./proofs/RA\_dod\_15/} (remember to create the directory, otherwise nothing will be saved).
The output we obtain is the following:

\begin{lstlisting}[basicstyle=\small, language=Python]
Inizializing...
: M has volume 17.2248... and homology Z/513
: Computing Turaev torsion drilling...
1: T(1, 2) has volume 14.2777... and homology Z/329
1: Computing Turaev torsion drilling...
11: The manifold T1 filled with (-1, 2) is [s345(-1,3)], its L-space value is 1
12: T1(0, 1) has volume 8.96606... and homology Z/143
12: Computing Turaev torsion drilling...
121: The manifold T12 filled with (1, 1) is [v3245(1,2)], its L-space value is 1
122: T12(3, 5) has volume 6.30690... and homology Z/59
122: T12(3, 5) is t12195(-1,-3), whose L-space value is known to be 1
2: T(1, 3) has volume 13.8548... and homology Z/237
2: Computing Turaev torsion drilling...
21: The manifold T2 filled with (0, 1) is [v2876(-1,2)], its L-space value is 1
22: T2(-1, 3) has volume 9.69817... and homology Z/66
22: Computing Turaev torsion drilling...
221: T22(-2, 3) has volume 6.51216... and homology Z/87
221: T22(-2, 3) is o9_36980(1,2), whose L-space value is known to be 1
222: T22(-3, 4) has volume 6.38898... and homology Z/94
222: T22(-3, 4) is o9_34893(-3,2), whose L-space value is known to be 1

True
\end{lstlisting}

Since the value of \texttt{val} is \texttt{True}, we conclude that the given manifold is an $L$-space.
The idea of the algorithm is the following: we start with the manifold $M$. We drill it and we obtain $T$. By filling $T$ we obtain $M1$ and $M2$. By drilling $M1$ we obtain $T1$. By filling $T1$ we obtain $M11$ and $M12$, and so on.

At the beginning of each line there is a number, which represents which \qhs{} we are referring to. For example, Line 7 starts with 12; this means we are referring to $M12$. 

A \emph{drilling-filling tree} is a tree where each node is either a \qhs{} or a \qht{}, and we have a labeled edge between $A$ and $B$ if $B$ is obtained by Dehn-filling $A$, and the edge is labelled with the coefficients of the filling.

From this output we discover the existence of a drilling-filling tree as in Figure \ref{figure:drillingfilling15}, that can be used to prove, through Theorem \ref{theorem RR}, that the given manifold is an $L$-space.
\begin{figure}[H]
\centering
\includegraphics[scale=0.65]{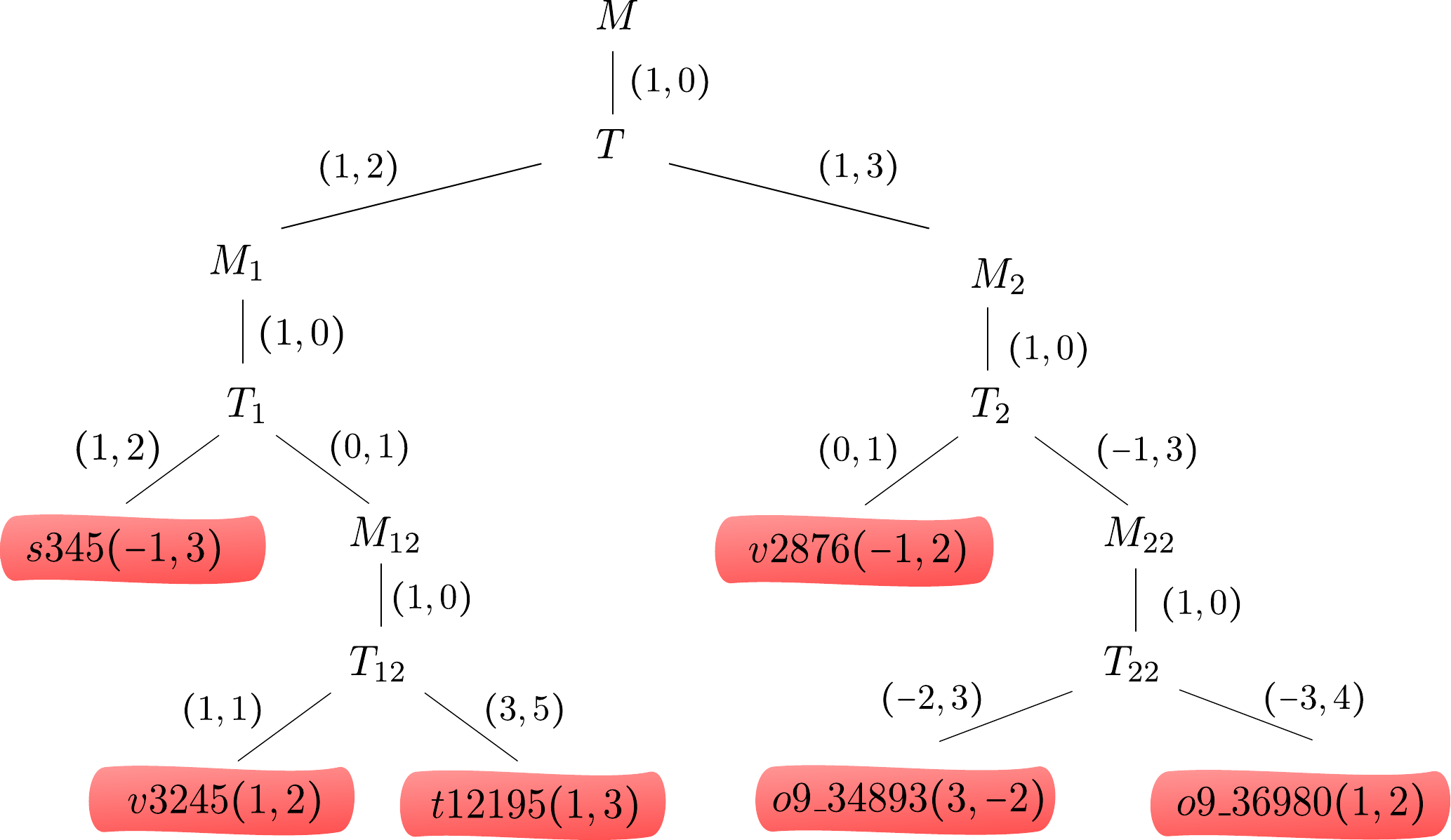}
\nota{One example of filling-drilling tree. The manifolds that are $L$-spaces belonging to the Dunfield census are highlighted in red. In this case M=CubicalOrientableClosedCensus(betti=0)[15].}
\label{figure:drillingfilling15}
\end{figure}

When $\sfrac{1}{0}$ is contained in $\iota^{-1}(D_{>0}^\tau)$, we have two possible intervals to work with. In this case the output looks like the following (we extract some lines from a bigger output):
\begin{lstlisting}[basicstyle=\small, language=Python]
2211: T221(-2, 3) has volume 7.91082... and homology Z/2 + Z/2 + Z/12
2211: Computing Turaev torsion drilling...
22111: T2211(3, 4) has volume 7.32772... and homology Z/2 + Z/2 + Z/8
22111: Computing Turaev torsion drilling...
22111: Double interval..
22111A1: T22111(1, 2) has volume 9.49189... and homology Z/2 + Z/2 + Z/8
22111A1: Computing Turaev torsion drilling...
22111A11: T22111A1(2, 1) has volume 11.0403... and homology Z/2 + Z/42
22111A11: Computing Turaev torsion drilling...
22111A111: The manifold T22111A11 filled with (0, 1) is [v2553(-4,1)], its L-space value is -1
22111B1: T22111(8, 5) has volume 8.92932... and homology Z/4 + Z/24
22111B1: Computing Turaev torsion drilling...
22111B11: T22111B1(2, 3) has volume 9.26762... and homology Z/108
22111B11: Computing Turaev torsion drilling...
\end{lstlisting}

At Line 5, the code is telling us that $T22111$, obtained by drilling $M22111$, has two intervals we want to work with. The fillings on the first of these intervals will have labels starting with $22111A$ and the ones on the second one will have labels starting with $22111B$.
From this output we discover the existence of a drilling-filling tree as in Figure \ref{figure:drillingfillingAB}. When we find a double interval, we have two possible intervals in $Sl(Y)$ where we can look for $L$-spaces. If for some filling in the first one the algorithm returns \texttt{False}, we go on looking in the other one.

\begin{figure}[H]
\centering

  \includegraphics[scale=0.65]{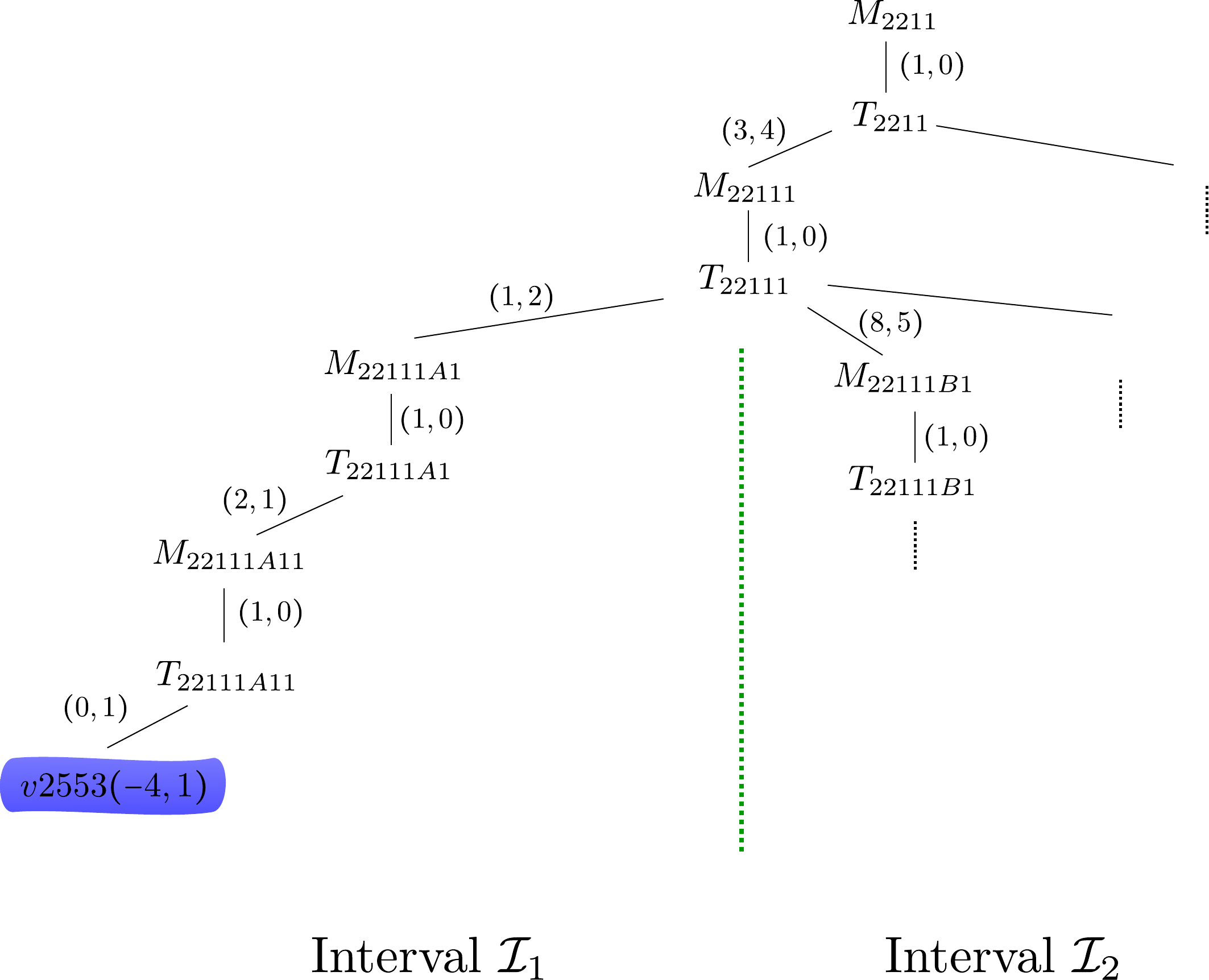}

\nota{A portion of a filling-drilling tree. In this case the slope $\sfrac{1}{0}$ is contained in $\iota^{-1}(D_{>0}^\tau)$ and therefore there are two intervals $\mathcal{I}_1$ and $\mathcal{I}_2$ in $Sl(T_{22111})$ to work with. The manifold $v2553(-4,1)$ is highlighted in blue since it belongs to the Dunfield census and it is not an $L$-space.}
\label{figure:drillingfillingAB}
\end{figure}

\end{document}